\documentclass[english]{amsart}

\usepackage[all,cmtip]{xy}
\usepackage{hyperref}
\usepackage{blindtext}

\usepackage{geometry}
\usepackage{amsmath,amssymb,amscd,amsfonts}

\usepackage{babel}
\usepackage{amstext}
\usepackage{amsmath}
\usepackage{amsfonts}
\usepackage{latexsym}
\usepackage{ifthen}

\usepackage[all,cmtip]{xy}
\xyoption{all}
\usepackage{enumitem}
\setlist[enumerate]{label=(\thethm.\arabic*), before={\setcounter{enumi}{\value{equation}}}, after={\setcounter{equation}{\value{enumi}}}}

\newcommand{\R}{\mathbb{R}}
\newcommand{\CC}{\mathbb{C}}
\newcommand{\Q}{\mathbb{Q}}

\newcommand{\Z}{\mathbb{Z}}

\newcommand{\ddbar}{\partial\bar{\partial}}

\newcommand{\wh}{\widehat}

\newcommand{\Dom}{\mathrm {Dom}}

\newcommand{\cC}{\mathcal{C}}
\newcommand{\cH}{\mathcal{H}}
\newcommand{\cP}{\mathcal{P}}

\newcommand{\cG}{\mathcal{G}}

\renewcommand{\O}{\mathcal{O}}

\newcommand{\ep}{\varepsilon}
\renewcommand{\epsilon}{\varepsilon}

\renewcommand{\ker}{\mathrm{Ker} \,}

\newcommand{\ol}{\overline}

\renewcommand{\leq}{\leqslant}
\renewcommand{\geq}{\geqslant}

\newcommand{\Div}{\mathrm{Div}}
\newcommand{\Id}{\mathrm{Id}}
\newcommand{\Imm}{\mathrm{Im} \,}

\newcommand{\Ker}{\mathrm{Ker}}
\newcommand{\Res}{\mathrm{Res}}

\newcommand{\dbar}{\bar \partial}

\newtheorem{thm}{Theorem}[section]
\newtheorem{lemme}[thm]{Lemma}
\newtheorem{proposition}[thm]{Proposition}
\newtheorem{quest}[thm]{Question}

\newtheorem{exemple}[thm]{Example}

\newtheorem{defn}[thm]{Definition}
\newtheorem{cor}[thm]{Corollary}

\newtheorem{remark}[thm]{Remark}

\numberwithin{equation}{thm}

\title{$\ddbar$-lemmas and a conjecture of O. Fujino}
\author{Junyan CAO, Mihai PAUN}

\address{Université Côte d’Azur, CNRS, LJAD, France}
\email{junyan.cao@unice.fr}

\address{Universit\"at Bayreuth, Mathematisches Institut, Lehrstuhl Mathematik VIII, Universit\"atsstrasse 30, D-95447, Bayreuth, Germany}
\email{mihai.paun@uni-bayreuth.de}

\begin{document} 

\maketitle


\section{Introduction}

\noindent Let $(X, \omega_X)$ be a $n$-dimensional compact Kähler manifold, and let $(F, h_F)$ be a holomorphic line bundle on $X$ endowed with a Hermitian metric $h_F$. If the curvature form 
$\displaystyle i\Theta(F, h_F)\geq 0$ is semi-positive, then for every $(n, q)$-form $v$ with values in $F$ 
such that 
$$\dbar v= 0, \qquad \int_X\langle[i\Theta(F, h_F), \Lambda_{\omega_X}]^{-1}v, v\rangle dV< \infty$$ and whose $L^2$ norm is finite one can solve the equation 
\begin{equation}\label{i1}
\dbar u= v,
\end{equation}
where $u$ is an $F$-valued $(n, q-1)$-form whose $L^2$--norm is finite. Moreover, the solution $u$ of \eqref{i1} which is orthogonal to 
$\Ker(\dbar)$ satisfies very precise $L^2$ estimates, cf. \cite{bookJP}. 
\smallskip

\noindent Consider next an snc divisor $E$ on $X$, and the corresponding space of forms with log poles along $E$.  Hodge theory results in \cite{Del}, \cite{EV86}
show that for any logarithmic $(n, q)$-form $v$
$$\dbar v= 0, \qquad v= \partial w$$
(where the first equality holds pointwise in the complement of the support of $E$ and $w$ is a $(n, q-1)$ log form) there exists a form $u$ with log poles along $E$ such that 
\begin{equation}\label{i2}
\dbar u= v,
\end{equation}
holds on $X\setminus E$. This is the log-version of the familiar $\ddbar$ lemma in Hodge theory. 
\smallskip

\noindent We remark that one could try to obtain \eqref{i2} as consequence of the $L^2$-theory, i.e. by interpreting 
$v$ as $(n, q)$-form with values in $F:= \O(E)$, the line bundle corresponding to $E$. However, the only natural, semi-positively curved metric $h_F$ we have at hand is singular along $E$, and the $L^2$ condition we have mentioned above means that we impose $v$ to vanish on $E$ --in conclusion, we would get a much weaker result. 
\medskip

\noindent In this article we are aiming at the following general vanishing theorem, which loosely speaking interpolates between Hodge and $L^2$-theory, respectively.

\begin{thm}\label{ddbar}
	Let $X$ be a $n$-dimensional compact Kähler manifold and let $E$ be a snc divisor on $X$, $s_E$ be the canonical section of $E$. Let $(L, h_L)$ be a holomorphic line bundle on $X$ such that
		$$i\Theta_{h_L} (L) = \sum q_i [Y_i] + \theta_L,$$ where $q_i \in ]0,1[ \cap \mathbb Q$, $E +\sum Y_i$ is snc, and the form $\theta_L$ is smooth, semi-positive.
Let $\lambda$ be a $\dbar$-closed smooth $(n,q)$-form with value in $L+E$. If there exists $\beta_1$ and $\beta_2$, two $L$-valued $(n-1, q)$-form and $(n-1, q-1)$-form with log poles along $E+\sum Y_i$ respectively, such that
\begin{equation}\label{yes!}
\frac{\lambda}{s_E} = D' _{h_L}\beta_1 + \theta_L \wedge \beta_2 \qquad\text{ on } X\setminus (E+\sum Y_i) ,
\end{equation}	
then the form $\lambda$ is $\dbar$-exact, i.e, 
	the class $[\lambda]=0 \in H^q (X, K_X +L+E)$.
\end{thm}

\medskip

\noindent As consequence of the methods developed in this paper (especially Theorem \ref{ddbar}) and in \cite{CPcan} we obtain the following injectivity statement, conjectured by O.~Fujino long ago (cf. \cite{Fuj17}).

\begin{thm}\cite[Conj 2.21]{Fuj17}\label{fujino}
Let $X$ be a compact Kähler manifold and let $E =\sum E_i$ be a simple normal crossing divisor on $X$. We consider a semi-positive line $F$
bundle on $X$ and a holomorphic section $s$ of $mF$ (for some $m\in\mathbb N$). We assume moreover that the zero locus of $s$ contains no lc centres of the lc pair $(X, E)$. Then the multiplication map
induced by the tensor product with $s$
\begin{equation}\label{inject}
	s\otimes:  \qquad H^q (X, K_X + E + F) \rightarrow H^q (X, K_X +E + (m+1)F)
	\end{equation}
is injective for every $q\in \mathbb N$.
\end{thm}

\noindent We refer to \cite{Mat19, CC22} for some special cases of the above result.
\medskip

\noindent Given the results obtained in \cite{Mat18}, it would be natural to adapt our method to the case of a singular metric $h_F$ (with arbitrary singularities). More precisely, we propose the following question.

\begin{quest}
Let $E =\sum E_i$ be a simple normal crossing divisor on a compact Kähler manifold $X$ and let $F$ be a holomorphic line bundle on $X$, endowed with a possible singular metric $h_F$ such that $\displaystyle i\Theta_{h_F} (F) \geq 0$. We suppose that $\mathcal{I} (h_F) = \mathcal{I} (h_F \cdot e^{-(1-\ep) \log |s_E|^2})$ for every $\ep >0$. 

Let $s$ be a (holomorphic) section of  $m F$ (for some $m\in\mathbb N$) such that $\displaystyle |s|_{h^m_F} <+\infty$ point-wise and the zero locus of $s$ contains no lc centres of the lc pair $(X, E)$. Then the multiplication map
induced by the tensor product with $s$
\begin{equation}\label{injectq}
	s\otimes:  \qquad H^q \big(X, K_X \otimes E \otimes F \otimes \mathcal{I} (h)\big) \rightarrow H^q \big(X, K_X \otimes E \otimes (m+1)F \otimes \mathcal{I} (h^{m+1})\big)
\end{equation}
is injective for every $q\in \mathbb N$.
\end{quest}
\medskip

\noindent Theorem \ref{ddbar} can be seen as a generalisation of the familiar $\ddbar$ lemma. Hence it should be no surprise that Hodge theory plays a key role in its proof.
The corresponding versions of Hodge decomposition theorem needed here, for metrics with Poincar\'e or conic singularities, respectively will be established in Section \ref{hodge}. We refer to \cite{Naumann} and the references therein for statements with a similar flavour. 
\smallskip

\noindent  
In this paper we introduce and study the notion of a $(p, q)$-current with values in a singular Hermitian bundle 
$(L, h_L)$, cf. Definition \ref{currents} in Section \ref{current}. The terminology refers to the fact that under the hypothesis of Theorem \ref{ddbar} we are forced to work with a K\"ahler metric with conic singularities on $X$; moreover, the model for the type of currents we have to manipulate here is provided by the RHS of \eqref{yes!}. 

\noindent The Green operator of such a current $T$ is defined by duality
\begin{equation}
\int_X \cG(T) \wedge \ol{\star \phi} := \int_{X}T\wedge\ol{\star \cG (\phi)},
\end{equation}
for any  $L$-valued $(p, q)$-form $\phi$, smooth in conic sense --for this important notion, as well for the definition of $\cG (\phi)$ we refer to 
Sections \ref{hodge} and \ref{current}. 

\noindent As consequence of our version of Hodge decomposition in Section {hodge}, we have a decomposition of the current $T$ which parallels the case of differential forms 
\begin{equation}
	T= \cH(T)+ \Delta''\big(\cG(T)\big).
\end{equation} 
where $\cH(T)$ is the projection of $T$ on the space of harmonic, $L$-valued forms of type $(p,q)$. 
We call this equality
\emph{Kodaira-de Rham decomposition of the current} $T$, by analogy with the untwisted case treated in \cite{KdR}. Together with the 
properties of the current $\cG(T)$ and an important lemma from \cite{LRW}, this is a key technique in the proof of Theorem \ref{ddbar},
and interesting in its own right. 

\medskip

\noindent{\bf Acknowledgments.} During part of the preparation of this article we have enjoyed the hospitality and excellent working conditions offered by the \emph{Freiburg Institut for Advanced Study}. At FRIAS we had the opportunity to meet many visitors, but the math exchanges we have had with S.~Kebekus and C. Schnell have been exceptionally fruitful. We thank A. Höring, J. Lott, R. Mazzeo and Y.~Rubinstein for valuable discussions. JC thanks the Institut Universitaire de France and the A.N.R JCJC project Karmapolis (ANR-21-CE40-0010) for providing excellent working condition. Finally, our work was completed during MP's visit at \emph{Center for Complex Geometry} (in Daejeon, South Korea). Many thanks to J.-M. Hwang for the invitation and vibrant working atmosphere in this institute!

\medskip

\noindent This paper is organised as follows.
\tableofcontents


\section{A few results from $L^2$ Hodge theory}\label{hodge}

\subsection{Hodge decomposition for metrics with Poincar\'e singularities}
Let $X$ be a $n$-dimensional compact Kähler manifold, and let $(L, h_L)$ be a  
line bundle endowed with a (singular) metric $h_L= e^{-\varphi_L}$. The terminology in this section is as follows.
\begin{enumerate}
	\smallskip
	
	\item[(a)] We say that the metric $h_L$ has log poles, if its local weights can be written as follows
	\begin{equation}\nonumber
		\varphi_L\equiv \sum a_i\log |f_i|^2
	\end{equation}
	modulo a smooth function, where $a_i$ are rational numbers and $f_i$ are holomorphic functions.
	\smallskip
	
	\item[(b)] If the metric $h_L$ has log poles, then we will consistently write  
	$$\displaystyle i\Theta_{h_L}(L)= \sum a_i [Y_i]+ \theta_L$$ 
	where $\theta_L$ is a smooth form.
\end{enumerate}

\noindent We consider a modification $\pi: \wh X\to X$ of $X$ such that the support of the singularities of 
$\varphi_L\circ \pi$ is a simple normal crossing divisor $E$. As usual, we can construct $\pi$ such that its restriction to $\wh X\setminus E$ is a biholomorphism . We write 
\begin{equation}\label{eq30}
	\varphi_L\circ \pi|_{\Omega}\equiv \sum_{\alpha=1}^p e_\alpha\log |z_\alpha|^2
\end{equation}
modulo a smooth function. Here $\Omega\subset \wh X$ is a coordinate subset, and 
$(z_\alpha)_{\alpha=1,\dots, n}$ are coordinates such that
$E\cap \Omega= (z_1\dots z_p= 0)$.

\noindent Let $\wh \omega_E$ be a complete K\"ahler metric on $\wh X\setminus E$,
with Poincar\'e singularities along $E$, and let 
\begin{equation}\label{eq31}
	\omega_E:= \pi_\star(\wh \omega_E)
\end{equation}
be the direct image metric. We note that in this way $(X_0, \omega_E)$ becomes a 
complete K\"ahler manifold, where $X_0:= X\setminus (h_L= \infty)$. 
\smallskip

\noindent The following statement is well-known, but we recall the precise version we need here.
We denote by $s$ the section of a line bundle for which the support of its zero divisor coincides with the support of $(h_L= \infty)$.
\begin{lemme}\label{cutoff}
	There exist a family of smooth functions $(\mu_\ep)_{\ep> 0}$ with the following properties.
	\begin{enumerate}
		
		\item[\rm (a)] For each $\ep> 0$, the function $\mu_\ep$ has compact support in $X_0$, and 
		$0\leq \mu_\ep\leq 1$.
		
		\item[\rm (b)] The sets $(\mu_\ep= 1)$ are providing an exhaustion of $X_0$.
		
		\item[\rm (c)] There exists a positive constant $C> 0$ such that 
		we have $$\displaystyle \sup_{X_0}\left(|\partial \mu_\ep|_{\omega_E}^2+ |\ddbar \mu_\ep|^2_{\omega_E}\right)\leq C.$$
	\end{enumerate} 
\end{lemme}

\begin{proof}
	It is enough to obtain the corresponding statement on $\wh X$, so that the divisor $E$ is snc. Then we take
	\begin{equation}\label{cut1}
		\mu_\ep= \rho_\ep\Big(\log\log\frac{1}{|\wh s|^2}\Big)
	\end{equation}
	where $\rho_\ep$ is equal to 1 on the interval $[1, \ep^{-1}]$ and is equal to zero on
	$[1+ \ep^{-1}, \infty]$. Also, we denote by $|\wh s|$ the inverse image of the norm 
	of the section $s$ with respect to an arbitrary, smooth metric.
	Actually we can also impose the condition that 
	\begin{equation}\label{cut2}
		\max_{j= 1,\dots N}\sup_{\ep> 0}\sup_{\R_+}|\rho^{(j)}_\ep|< C_N<\infty,
	\end{equation}
	for any positive integer $N$.
	The properties (a)-(c) are then verified by a simple computation.
\end{proof}

\medskip

\noindent In this context we have the following statement, which is a slight 
generalization of the usual result in $L^2$ Hodge theory.

\begin{thm}\label{Hodge}
	Consider a line bundle $(L, h_L)\to X$ endowed
	with a metric $h_L$ such that the requirement {\rm (a)} above is satisfied and assume that $\theta_L \geq 0$ on $X$.  Let $(X_0, \omega_E)$ be the corresponding complete K\"ahler manifold 
	cf. \eqref{eq31}. Then the following assertions are true.
	\begin{enumerate}
		\item[\rm (i)] We have the Hodge decomposition for $(n, q)$ forms, i.e.
		$$
		L^2_{n, q}(X_0, L)= {\mathcal H}_{n,q}(X_0, L)\oplus
		\Imm \dbar\oplus \Imm \dbar^\star,$$
		where   ${\mathcal H}_{n,q}(X_0, L)$ is the space of $L^2$ $\Delta^{''}$-harmonic $(n,q)$-forms with respect to $(\omega_E, h_L)$.
		
		\item[\rm (ii)] Let $u$ be an $L$-valued $\dbar$-closed $L^2$-form of type $(n,q)$ on $X_0$, and
		assume that we
		have $u= D^\prime _{h_L} w$, where $w\in L^2_{n-1, q}(X_0, L)$. Then there exists $v\in L^2_{n, q-1}(X_0, L)$ such
		that $u= \dbar v$.
	\end{enumerate}
\end{thm}  

\medskip

\noindent The proof of Theorem~\ref{Hodge}, which we give below, makes use of the
following statement which is the $\dbar$-version of the Poincar\'e inequality established in \cite{Auv17}. 

\begin{thm}\label{prop:2}
	Let $p\leq n$ be an integer.  There exists a positive constant $C> 0$ such
	that the following inequality holds
	\begin{equation}\label{eq32}
		\int_{X_0}|u|^2_{\omega_E}e^{-\varphi_L}dV_{\omega_E}\leq C \int_{X_0}|\dbar u|^2_{\omega_E}e^{-\varphi_L}dV_{\omega_E}
	\end{equation}
	for any $L$-valued form $u$ of type $(p, 0)$ which belongs to the domain of $\dbar$ and which is orthogonal to the
	space of $L^2$ holomorphic $(p,0)$--forms. Note that here $\theta_L$ is arbitrary.
\end{thm}   
\begin{proof}[Proof of Theorem~\ref{prop:2}]
	The first observation is that we can assume from the very beginning that the polar set
	$h_L= \infty$ has snc support. Indeed, via the map $\pi$ the hypothesis and the conclusion of 
	our Poincar\'e inequality transform as follows.
	\smallskip
	
	\noindent $\bullet$ The form 
	\begin{equation}
		\wh u= \pi^\star u
	\end{equation}
	on $\wh X$ is $\pi^\star L$-valued and $L^2$ with respect to $\wh \omega_E$ and $\pi^\star h_L$.
	\smallskip
	
	\noindent $\bullet$ The form $\wh u$ is orthogonal on the space of $L^2$--harmonic $(p,0)$ forms with values in $\pi^\star L$. Indeed, given any $\pi^\star L$-valued form of 
	$(p, 0)$-type $\wh \gamma$ on $\wh X\setminus E$ there exists a form $\gamma$ on 
	$X\setminus Z$ such that 
	\begin{equation}
		\wh \gamma= p^\star \gamma
	\end{equation} 
	because $p:\wh X\setminus E\to X\setminus Z$ is a biholomorphism. Also, $p$ establish an 
	isometry between $(\wh X\setminus E, \wh \omega_E)$
	and $(X\setminus Z, \omega_E)$ respectively. The same is true for the pairs
	$(p^\star L, p^\star h_L)$ and $(L, h_L)$, so our claim is clear.
	\medskip
	
	\noindent This means that it is enough to assume that $\wh X= X$, i.e. the metric $\omega_E$ has 
	Poincar\'e singularities along an snc divisor $E$ which coincides with the support of the analytic set $h_L= \infty$.
	\smallskip
	
	\noindent The following statement represents an important step towards the proof of Theorem \ref{prop:2}.
	
	\begin{proposition}\label{poinc} Let $\Omega$ be some co-ordinate open set of $X$ and
let $\tau$ be a $L$-valued $(p, 0)$-form with compact support 
		in $\Omega$. Then we have 
		\begin{equation}\label{eq57}
		\int_\Omega|\tau|^2_{\omega_E} e^{-\varphi_L}dV_{\omega_E}\leq C\int_\Omega|\dbar\tau|^2_{\omega_E}e^{-\varphi_L}dV_{\omega_E}
		\end{equation}
		where $C$ is a   constant.
	\end{proposition}
	\medskip
	
	\noindent The arguments in \cite{Auv17}, Lemma 1.10 correspond to the case $p=0$ and $\varphi_L=0$. For the general case we need the
	following auxiliary statement.
	
	\begin{lemme}\label{fourier} Let $\delta$ be an arbitrary real number. Then there exists $x_\delta< 1$ such that the following holds.
		For any smooth function $f$ defined on the closure of the unit disk $\mathbb D$ in $\mathbb C$, such that 
		there exists $\ep> 0$ with the property that $f(z)= 0$ on the set $\{ |z| \leq \ep \} \cup  \{|z|\geq 1- x_\delta \}$, we have 
		\begin{equation}\label{d1}\int_{\mathbb D}|f|^2 \frac{d\lambda(z)}{|z|^{2\delta}} \leq C \int_{\mathbb D}\left| \frac{\partial f }{\partial \ol z}\right|^2 |z|^{2- 2\delta} \log^2 |z| d\lambda(z)
		\end{equation}
		as well as
		\begin{equation}\label{d2}
			\int_{\mathbb D}|f|^2 \frac{d\lambda(z)}{|z|^{2+ 2\delta}\log^2 |z|} \leq C \int_{\mathbb D}\left| \frac{\partial f }{\partial \ol z}\right|^2 \frac{d\lambda(z)}{|z|^{2\delta}}.
		\end{equation}
	\end{lemme}

	\begin{proof}[Proof of Lemma  \ref{fourier}] Let $x_\delta< 1$ be a positive real number whose expression will be determined in what follows. Let $v:[0,1]\to \mathbb C$ be a smooth function whose support is contained in the interval $[\ep, x_\delta]$ for some positive, small enough $\ep$. 
		
		\noindent We claim that the following inequalities hold
		\begin{equation}\label{cc1}
			\int_0^1 |v|^2 t^kdt \leq C \int_0^1 |v' |^2 t^{k+2}\log^2t dt
		\end{equation}
		as well as
		\begin{equation}\label{c1bis}
			\int_0^1 |v|^2 t^k\frac{dt}{\log^2t} \leq C \int_0^1 |v' |^2 t^{k+2}dt
		\end{equation}
		where $k$ is any element of the set $\delta+ \mathbb Z$, and $C$ is a numerical constant (in particular independent of $v, \ep, k, ...$).
		
		\noindent We first check that \eqref{cc1} holds true: we have 
		\begin{equation}\label{e1}
			-\int_0^1 |v|^2 t^{k+1}\log^2 d\left(\frac{1}{\log t}\right)= \int_0^1\frac{1}{\log t} \frac{d}{dt}\left(v^2 t^{k+1}\log^2t\right)dt  
		\end{equation}
		and this equals
		\begin{equation}\label{e2}
			2\int_0^1(\ol v v'+ v\ol{v'}) t^{k+1}\log tdt + \int_0^1\frac{t^kv^2}{\log t}\left((k+1)\log^2t+ 2\log t\right)dt.
		\end{equation}
		Now two things can happen: either $k+1=0$, and then the second term in \eqref{e2} is equal to twice the quantity we are interested in, or $k+1\neq 0$. Note that in the latter case it is possible to fix $x_\delta$ uniformly with respect to $k$ so that 
		\begin{equation}\label{e201}|(k+1)\log^2t+ 2\log t | \geq 2 |\log t|\end{equation}
		for any $t\in [0, x_\delta]$ precisely because $k\in \delta+ \mathbb Z$.
		
		\noindent As a consequence, we have 
		\begin{equation}\label{e200}
			\left|\int_0^1 (\ol v v'+ v\ol{v'}) t^{k+1}\log tdt\right| \geq  \int_0^1 t^kv^2 dt.
		\end{equation}
		for any $v$ of compact support in $]0, x_\delta[$. 
		
		\noindent Moreover, we have 
		\begin{equation}\label{e3}
			\left|\int_0^1(\ol v v'+ v\ol{v'}) t^{k+1}\log tdt\right|^2 \leq \int_0^1|v'|^2 t^{k+2}\log^2tdt \int_0^1|v|^2 t^{k}dt.
		\end{equation}
		by Cauchy-Schwarz, so the inequality \eqref{cc1} is settled. 
		\medskip
		
		\noindent The arguments for \eqref{c1bis} are completely similar: we have 
		\begin{equation}\label{c11}
			\int_0^1 |v|^2 t^k\frac{dt}{\log^2t}=  \int_0^1 |v|^2 \frac{t^{k+1}}{\log^2 t}d\log t
		\end{equation}
		which equals 
		\begin{equation}\label{c12}
			-\int_0^1 {\log t}\frac{d}{dt}\left(|v|^2 \frac{t^{k+1}}{\log^2 t}\right)dt= \int_0^1 
			\frac{|v|^2t^{k}}{\log^2t}\left(2+ (k+1)\log \frac{1}{t}\right) dt -
			\int_0^1\frac{t^{k+1}}{\log t}(v'\ol v+ v\ol {v'})dt.
		\end{equation}
		and two things can happen. If $k+1\geq 0$, then we have $\displaystyle 2+ (k+1)\log \frac{1}{t}\geq 2$ for any $t\in [0,1]$ and \eqref{c1bis} follows.
		For each element $k\in \delta+ \mathbb Z$ such that $k+1< 0$ we can assume that 
		\begin{equation}\label{e2011}(k+1)\log \frac{1}{t}+ 2 \leq \frac{1}{2}\end{equation}
		for any $t\in [0, x_\delta]$ by imposing a supplementary (but uniform) condition on $x_\delta$ if necessary, and we proceed exactly as we did for \eqref{cc1}.
		
		\smallskip
		Therefore, we obtain \eqref{cc1} and \eqref{c1bis} for any 
		$k\in \delta+ \Z$.  
		\medskip
		
		\noindent  Lemma \ref{fourier} follows from this: we consider the Fourier series 
		\begin{equation}\label{c5}
			f=\sum_{k\in \mathbb Z}a_k(t)e^{\sqrt{-1}k\theta}
		\end{equation}
		of $f$, and then we have 
		\begin{equation}\label{c6}
			\frac{\partial v }{\dbar z}= \sum_{k\in \mathbb Z}\left(a_k'(t)- \frac{k}{t}a_k(t)\right)e^{\sqrt{-1}k\theta}.
		\end{equation}
		The identity
		\begin{equation}\label{c7}
			t^k\frac{d}{dt}\left(\frac{a_k}{t^k}\right)= a_k'(t)- \frac{k}{t}a_k(t)
		\end{equation}
		reduces the proof of our statement to inequalities of type \eqref{cc1} and \eqref{c1bis} which are already established. In conclusion, Lemma \ref{fourier} is completely established.
	\end{proof}
	\medskip
	
	\noindent We now return to the proof of Proposition \ref{poinc}. 
	
	\begin{proof}[Proof of Proposition \ref{poinc}]
		We can suppose that the intersection $\Omega \cap E$ is of type $z_1\cdots z_r= 0$. 
		
		We first consider the case when $\tau$ is of compact support in $\Omega \setminus E$.  
		Then $\tau$ can be written as sum of forms of type
		\begin{equation}\label{d111}
			\tau_I:= f_Idz_I .
		\end{equation}
After changing the orders of the index,	there are two cases:  $I\cap \{1,\dots, r\}= \{1,\dots, p\}$ for some $p$ or $I\cap \{1,\dots, r\}= \emptyset.$
		
		\noindent In the first case we have
		\begin{equation}\label{d222}
			\frac{|\tau_I|_g^2e^{-\varphi_L}dV_g}{d\lambda}= \frac{|f_I|^2e^{-\varphi_L}}{\prod_{\alpha=p+1}^r|z_\alpha|^2\log^2|z_\alpha|^2}
		\end{equation}
		where $\displaystyle e^{-\varphi_L}= \frac{1}{\prod_{i=1}^r|z_i|^{2\delta_i}}$
		and in the second case this is the same with $p=0$.  
		In \eqref{d222} we denote by $g$ the model Poincaré metric
		$$\sum_{i=1}^r\frac{\sqrt{-1}dz_i\wedge d\ol z_i}{|z_i|^2\log^2|z_i|^2}+ \sum_{i=r+1}^n{\sqrt{-1}dz_i\wedge d\ol z_i}.$$
		\smallskip
		
		\noindent Now for the $\dbar \tau_I$ we have
		\begin{align}
			\frac{|\dbar \tau_I|_g^2e^{-\varphi_L}dV_g}{d\lambda}= & \sum_{i=1}^r\left|\frac{\partial f_I}{\partial \ol z_i}\right|^2\frac{|z_i|^2\log^2|z_i|^2e^{-\varphi_L}}{\prod_{\alpha=p+1}^r|z_\alpha|^2\log^2|z_\alpha|^2}\nonumber \\ + & \label{d3} \sum_{i=r+1}^n\left|\frac{\partial f_I}{\partial \ol z_i}\right|^2\frac{e^{-\varphi_L}}{\prod_{\alpha=p+1}^r|z_\alpha|^2\log^2|z_\alpha|^2} \\
			\nonumber 
		\end{align}
		and the observation is that for every $\xi\in \Omega$ there exist some index $j\in \{1,\dots, r\}$ for which
		the support condition we require in Lemma \ref{fourier} is satisfied for 
		\begin{equation}\label{d4}
			f(z):= f_I(\xi_1,\dots \xi_{j-1}, z, \xi_{j+1},\dots \xi_n).
		\end{equation}
		We use \eqref{cc1} or \eqref{c1bis} according to the case where $j\geq p$ in \eqref{d3} or not.
		\medskip

	In general, if $\tau$ is of compact support in $\Omega$, we use the cut-off function $\mu_\ep$ in Lemma \ref{cutoff}. We apply the above estimation for $\mu_\ep \tau$. Note that the constant $C$ in Lemma \ref{fourier} is uniform with respect to $\ep$.
	Proposition \ref{poinc} is therefore completely proved by letting $\ep \to 0$.
	\end{proof}
	\bigskip
	
	\noindent Now we finish the proof of Theorem \ref{prop:2}, which  follows the arguments in \cite[Lemma 1.10]{Auv17}:  if a positive constant as in \eqref{eq32} does not exists, then we obtain a sequence $u_j$ of 
	$L$-valued forms of type $(p,0)$ orthogonal to the space of holomorphic forms such that 
	\begin{equation}\label{eq61}
		\int_X|u_i|^2_{\omega_E}e^{-\varphi_L}dV_{\omega_E}= 1, \qquad
		\lim_i\int_X|\dbar u_i|^2_{\omega_E}e^{-\varphi_L}dV_{\omega_E}= 0
	\end{equation}
	It follows that the weak limit $u_\infty$ of $(u_i)$ is holomorphic. On the other hand, each $u_i$ 
	is perpendicular to the space of holomorphic forms, so it follows
	that $u_\infty$ is equal to zero. Moreover, by Bochner formula combined with Proposition \ref{poinc} we can assume that $u_i$ 
	converges weakly in the Sobolev space $W^1$, so strongly in $L^2$ to zero. In particular we have 
	\begin{equation}\label{eq611}
		u_i |_K\to 0
	\end{equation}
	in $L^2$ for any compact subset $K\subset X\setminus E$.
	
	The last step in the proof is to notice that the considerations above contradict the fact that the $L^2$ norm of each $u_i$ is equal to one. This is not quite immediate, but is precisely as the end of the proof of Lemma 1.10 in
	\cite{Auv17}, so we will not reproduce it here. The idea is however very clear: 
	we decompose each $u_j= \chi u_j+ (1-\chi)u_j$ where $\chi$ is a cutoff function which is equal to 1 in $U$ and whose norm of the corresponding gradient is small. Then the $L^2$ norm of $\chi u_j$ is small
	by \eqref{eq57} and \eqref{eq61}. The $L^2$ norm of $(1-\chi)u_j$ is equally small 
	by \eqref{eq611}, and this is how we reach a contradiction.
\end{proof}
\medskip

\noindent We have the following consequences of Proposition \ref{poinc}. 

\begin{cor}\label{usecor} In the setting of Theorem \ref{Hodge}, we can find a positive constant $C> 0$ such that the following holds true.
	Let $v$ be a $L$-valued form of type $(n,q)$ for some $q\geq 1$. We assume that $v$ is $L^2$, and orthogonal 
	to the kernel of the operator $\dbar^\star$. Then we have
	\begin{equation}
		\int_{X_0}|v|^2_{\omega_E}e^{-\varphi_L}dV_{\omega_E}\leq C \int_{X_0}|\dbar^\star v|^2_{\omega_E}e^{-\varphi_L}dV_{\omega_E}.
	\end{equation}
\end{cor}
\smallskip

\begin{proof} To start with, we consider  
	the Hodge star $u:= \star v$ (of type $(n-q, 0)$ and values in $L$) is orthogonal to the space $\Ker(D' _{h_L})$. This can be seen as follows:
	let $\phi$ be a $L^2$-form of type $(n-q, 0)$ such that $D'_{h_L} \phi= 0$. Then we have
	$$\int_{X_0} \langle u, \phi \rangle e^{-\varphi_L} d V_{\omega_E} = \int_{X_0} \langle v, \star \phi \rangle e^{-\varphi_L} d V_{\omega_E} =0.$$
	since $\displaystyle \dbar^\star(\star \phi)= -\star D'_{h_L} (\phi)= 0$.
	
	\noindent We therefore have to show that there exists some constant $C> 0$ such that the inequality
	\begin{equation}\label{eq63}
		\int_{X_0}|u|^2_{\omega_E}e^{-\varphi_L}dV_{\omega_E}\leq 
		C\int_{X_0}|D'_{h_L} u|^2_{\omega_E}e^{-\varphi_L}dV_{\omega_E}.
	\end{equation}
	holds for any $L$-valued form $u$ of type $(n-q, 0)$ orthogonal to the kernel of $D' _{h_L}$. 
	
	\medskip
	
	As before, this is done contradiction: if we cannot find a constant $C$ as in 
	\eqref{eq63}, then we get a sequence of forms $u_i$ such that 
	\begin{equation}\label{van20}
		\|u_i\|^2= 1, \qquad \lim_{i\to\infty}\|D'_{h_L} (u_i)\|^2= 0,
	\end{equation}
	where the norms in \eqref{van20} are precisely the ones in \eqref{eq63}. It follows that any weak limit $u_\infty$ of the sequence $(u_i)$ must be zero, since it belongs to both $\Ker(D'_{h_L} )$ and to its orthogonal. Therefore we get 
	\begin{equation}\label{van21}
		u_i |_K\to 0
	\end{equation}
	in $L^2$ for any compact subset $K\subset X\setminus E$. 
	\smallskip
	
	\noindent On the other hand, for each index $i$ we have $D'^{\star} (u_i)= 0$. So Bochner formula reads
	\begin{equation}\label{van22}
		\int_{X_0}|\dbar u_i|^2_{\omega_E}e^{-\varphi_L}dV_{\omega_E}= \int_{X_0}|D' _{h_L} u_i|^2_{\omega_E}e^{-\varphi_L}dV_{\omega_E}+ 
		\int_{X_0}\langle [\theta_L, \Lambda_{\omega_E}]u_i, u_i\rangle e^{-\varphi_L}dV_{\omega_E}
	\end{equation}
	and since -by the assumption $\theta_L\geq 0$, we have 
	\begin{equation}\label{van23}
		\int_{X_0}\langle [\theta_L, \Lambda_{\omega_E}]u_i, u_i\rangle e^{-\varphi_L}dV_{\omega_E}\leq 0
	\end{equation}
	because of the type of $u_i$.  
	\smallskip
	
	\noindent It follows that $\displaystyle \lim_i \|\dbar u_i\|^2= 0$ and we obtain a contradiction exactly as in end of the proof of Theorem \ref{prop:2}.
\end{proof}
\medskip

\noindent The following is the "adjoint" version of the preceding corollary.  

\begin{cor}\label{usecorII} In the setting of Theorem \ref{Hodge}, there exists a positive constant $C> 0$ such that the following holds true.
	Let $v$ be a $L$-valued form of type $(n,q)$ for some $q\geq 1$. We assume that $v$ is $L^2$, and orthogonal 
	to the kernel of the operator $\dbar$. Then we have
	\begin{equation}\label{eq62}
		\int_{X_0}|v|^2_{\omega_E}e^{-\varphi_L}dV_{\omega_E}\leq C \int_{X_0}|\dbar v|^2_{\omega_E}e^{-\varphi_L}dV_{\omega_E}.
	\end{equation}
\end{cor}

\begin{proof}
	The arguments are completely identical: if $v$ is orthogonal on $\Ker(\dbar)$, then $u:= \star v$ is orthogonal to 
	the kernel of $D'^\star$. Thus the inequality to be established is 
	\begin{equation}\label{van24}
		\int_{X_0}|u|^2_{\omega_E}e^{-\varphi_L}dV_{\omega_E}\leq 
		C\int_{X_0}|D'^\star u|^2_{\omega_E}e^{-\varphi_L}dV_{\omega_E}.
	\end{equation}
	holds for any $L$-valued form $u$ of type $(n-q, 0)$ orthogonal to the kernel of $D'^\star$. This is (again) done by contradiction 
	--note that in this case $D' _{h_L} u= 0$-- , and we will skip the details.
\end{proof}

\medskip

\begin{proof}[Proof of Theorem~\ref{Hodge}]
	This statement is almost contained in \cite[chapter VIII, pages
	367-370]{bookJP}. Indeed, in the context of complete manifolds one has the
	following decomposition
	\begin{equation}
		L^2_{n, q}(X_0, L)= {\mathcal H}_{n,q}(X_0, L)\oplus \overline{\Imm \dbar}\oplus \overline{\Imm \dbar^\star}.
	\end{equation}
	We also know (see \emph{loc. cit.}) that the adjoints $\dbar^\star$ and
	$D^{\prime\star}$ in the sense of von Neumann coincide with the formal
	adjoints of $\dbar$ and $D^\prime$ respectively.
	
	It remains to show that the ranges of the $\dbar, \dbar^\star$-operators are closed with
	respect to the $L^2$ topology. In our set-up, this is a consequence of the
	particular shape of the metric $\omega_E$ at infinity (i.e. near the support
	of $\pi (E)$).  The image of $\dbar^\star$ is closed thanks to Corollary \ref{usecor} and the fact that the image of $\dbar$ is closed follows from Corollary \ref{usecorII}
	\medskip
	
	The assertion $(ii)$ is an easy consequences of $(i)$, together with standard considerations. 
\end{proof}
\medskip

\begin{remark}{\rm 
		It would be very interesting to know if the images of the operators $\dbar$ and $\dbar^\star$ are closed for $L$-valued forms of 
		arbitrary type, especially if the non-singular part $\theta_L$ of the curvature of $(L, h_L)$ has arbitrary sign. In our arguments, the requirement 
		$\theta_L\geq0$ is imposed by the -extensive- use of Bochner formula. We refer to \cite{EV86} and \cite{Fuj92} for results in this direction.}\end{remark}

\medskip

\noindent In the setting of the current section, the harmonic $(n,1)$-forms have the following properties.

\begin{lemme}\label{vanishing} In the setting of Theorem \ref{Hodge},
	let $\alpha$ be a $L^2$ $\Delta''$-harmonic form of type $(n,1)$ with values in $L$. We denote by $F:= \star \alpha$ its Hodge dual. Then the following hold.
	
	\begin{enumerate}
		
		\item[\rm (1)] We have 
		\begin{equation}\nonumber
			\dbar F=0\qquad \text{and } \qquad D' _{h_L} F=0.
		\end{equation}
		
		\item[\rm (2)] Let $\Omega\subset X$ be a coordinate open set, and let $f$ be any holomorphic function defined on $\ol \Omega$. We assume that the support of the set of zeroes of $f$ is contained in 
		$(h_L = \infty)$. 
		Then the form $\displaystyle F\wedge \frac{df}{f}$ is non-singular on $X$ (but it is not clear whether this form is $L^2$ with respect to $e^{-\varphi_L}$ or not).
		
		\item[\rm (3)] Let $\Omega\subset X$ be a coordinate open set, and let $g$ be any holomorphic function defined on $\ol \Omega$. Then we have  
		\begin{equation}\nonumber
			\int_\Omega |F\wedge dg|^2e^{-\varphi_L}< \infty.    
		\end{equation} 
	\end{enumerate}
\end{lemme}

\begin{proof}
	
	\noindent Note that given any $\Delta''$-harmonic $(n,1)$-form $\alpha$, 
	by Bochner formula combined with the fact that the curvature of $(L, h_L)$ is semi-positive,
	we have
	\begin{equation}\label{snc8-}
		D '^{\star} \alpha =0  \qquad\text{on } X \setminus (h_L = \infty) .
	\end{equation}
	So if we write $\displaystyle \alpha = \omega_E \wedge F$, it follows that
	\begin{equation}\label{dps}
		F \in H^0 \big(X, \Omega_X ^{n-1} \otimes L\big),
		\qquad \int_X|F|^2_{\omega_E}e^{-\varphi_L}dV_{\omega_E}< \infty
	\end{equation}
	thanks to the property \eqref{snc8-}. 
	Moreover, as $\displaystyle \dbar^* \alpha= 0$ it turns out that we have $D'_{h_L}  F=0$ as well. 
	We obtain thus $(1)$ and $(3)$.
	Note that \eqref{dps} is valid in a more general setting,
	cf. \cite{DPS01}, \cite{Wu19}.  
	\bigskip
	
	\noindent As for $(2)$, by our choice of the metric $\omega_E$, we show first that we have
	\begin{equation}\label{snc10}
		|df|^2 _{\omega_E} \sim  |f|^2 \log^2 |f|.  
	\end{equation}
	Indeed, assume that via the bimeromorphism $\pi: \wh X \rightarrow X$ the function $f$ corresponds locally to
	the function $\displaystyle \prod_{i=1}^r z_i^{p_i}$ (notations as in the beginning of this section). Then the metric $\omega_E$
	is quasi-isometric to the Poincaré metric with singularities along $z_1\dots z_r= 0$. The $\pi$-inverse image of $\displaystyle\frac{df}{f}$ is equal to
	$$\sum_{i=1}^rp_i\frac{dz_i}{z_i}
	$$
	and thus we have
	$$\left|\frac{\partial f}{f}\right|^2_{\omega_E}\simeq \sum p_i^2\log^2|z_i|^2$$
	and the RHS of this quantity is the same as $\log^2|f|^2$.
	
	\noindent Therefore we infer
	\begin{equation}\label{snc11}
		\int_{\Omega}  \left|F\wedge \frac{df}{f}\right|^2\frac{e^{-\varphi_L}}{\log^2 |f|} \leq 
		C\int_X {|F |^2_{\omega_E} }{}e^{-\varphi_L}dV_{\omega_E}.  
	\end{equation}
	Since the RHS of \eqref{snc11} is finite and $\varphi_L$ has non zero Lelong number over any component of $Div (f)$, it follows that $F \wedge df$ vanishes along $Div (f)$. 
\end{proof}	
\medskip

\noindent Consider the equation 
\begin{equation}\label{hodge1}
	\dbar u= D' _{h_L} w+ \Theta(L, h_L)\wedge \tau \qquad\text{ on } X\setminus \{h_L = \infty\}
\end{equation}
where $\tau, w$ and $D'w$ are $L^2$, and $D' _{h_L} w+ \Theta(L, h_L)\wedge \tau$ is $\dbar$-closed on  $X\setminus \{h_L = \infty\}$. We show next that 
one can solve it in a fair general context, and moreover obtain $L^2$ estimates for the solution with minimal norm. We refer to \cite[Thm A.5]{Wan17} for a similar argument in the non-singular case.

\begin{thm}\label{deprim} Let $(L, h_L)$ be a holomorphic line bundle on $X$ with a possible singular metric $h_L$ with analytic singularities along a subvarities $Z \subset X$ and whose curvature current is semi-positive.
	Consider a complete K\"{a}hler metric $(X\setminus Z, \omega_E)$.
	Let $w$ be an $L$-valued $(n-1, 1)$--form on $X\setminus Z$ such that $w$ and $D' _{h_L} w$ are in $L^2$. Let also $\tau$ be an $L$-valued $(n-1, 0)$--form on $X\setminus Z$ such that $\tau$
	and $\Theta(L, h_L)\wedge \tau$ are $L^2$.
	
	If $D' _{h_L} w+ \Theta(L, h_L)\wedge \tau$ is $\dbar$-closed on  $X\setminus Z$,
	then there exists a solution $u$ of the equation \eqref{hodge1} such that 
	\begin{equation}\label{hodge2}
		\int_X|u|^2e^{-\varphi_L}\leq \int_X|w|^2_{\omega_E}e^{-\varphi_L}dV_{\omega_E}
		- \int_{X\setminus Z}\langle[\Theta(L, h_L), \Lambda_{\omega_E}]\tau, \tau \rangle_{\omega_E}e^{-\varphi_L}dV_{\omega_E}
	\end{equation}
\end{thm}
\begin{proof}
	Let $\xi$ be a smooth $(n,1)$-form of compact support in $X\setminus Z$. Let $\xi =\xi_1 +\xi_2$ be its decomposition according to $\Ker(\dbar)$ and its orthogonal.
	
	\noindent By using the $L^2$-assumptions, we have 
	$$ \int_{X\setminus Z}\langle D'w+ \Theta(L, h_L)\wedge \tau , \xi \rangle_{\omega_E}e^{-\varphi_L}dV_{\omega_E} =
	\int_{X\setminus Z}\langle D'w+ \Theta(L, h_L)\wedge \tau , \xi_1 \rangle_{\omega_E}e^{-\varphi_L}dV_{\omega_E} .$$
	Then the semipositive curvature condition and the Bochner equality imply that
	\begin{equation}\label{hodge200} \left|\int_{X\setminus Z}\langle D'w , \xi_1 \rangle_{\omega_E}e^{-\varphi_L}dV_{\omega_E} \right|^2\end{equation}
	\noindent is smaller than
	
	\begin{equation}\label{hodge201}\int_X|w|^2_{\omega_E}e^{-\varphi_L}dV_{\omega_E} \cdot 
		\int_X|\dbar^\star \xi|^2_{\omega_E}e^{-\varphi_L}dV_{\omega_E}.
	\end{equation}
	For the term containing the curvature of $(L, h_L)$ we are using Cauchy inequality and we infer that 
	\begin{equation}\label{hodge202}\left|\int_{X\setminus Z}\langle  \Theta(L, h_L)\wedge \tau , \xi_1 \rangle_{\omega_E}e^{-\varphi_L}dV_{\omega_E}\right|^2\end{equation}
	is smaller than
	
	\begin{equation}\label{hodge203}\int_X\langle -[\Theta(L, h_L), \Lambda_{\omega_E}]\tau, \tau \rangle_{\omega_E}e^{-\varphi_L}dV_{\omega_E} \cdot 
		\int_X\langle[\Theta(L, h_L), \Lambda_{\omega_E}]\xi_1, \xi_1 \rangle_{\omega_E}e^{-\varphi_L}dV_{\omega_E}.
	\end{equation}  
	We are using again Bochner equality and we infer that 
	\begin{equation}\label{hodge204}\left|\int_{X\setminus Z}\langle  \Theta(L, h_L)\wedge \tau , \xi_1 \rangle_{\omega_E}e^{-\varphi_L}dV_{\omega_E}\right|^2
	\end{equation}
	is smaller than
	\begin{equation}\label{hodge205}\int_X\langle -[\Theta(L, h_L), \Lambda_{\omega_E}]\tau, \tau \rangle_{\omega_E}e^{-\varphi_L}dV_{\omega_E} \cdot  \int_X|\dbar^\star \xi|^2_{\omega_E}e^{-\varphi_L}dV_{\omega_E}.
	\end{equation}  
	Then the theorem is proved by the standard $L^2$-estimate argument.
\end{proof}
\medskip


\subsection{Hodge decomposition for metrics with conic singularities, I}\label{subsectconic}

\medskip

\noindent We discuss in this section the analog of Theorem 
\ref{Hodge} for metrics with conic singularities. This will be important for the applications we are aiming at.
There are two possibilities to deal with this sort of problems: either we approximate our metrics with complete ones (e.g. with Poincaré
singularities), or directly, by using local uniformisations. In this subsection we use the former approach and in the next subsection we present the later.

\noindent Since most of the arguments in the proof of Theorem \ref{Hodge} will be used here, we will skip a few details. In the last part of this section we offer a few comments about the $L^2$ estimates in this class of singularities. The setting of this subsection is as follows.
\smallskip

\noindent \textbf{Setting}: Let $X$ be a compact K\"ahler manifold, and let $\displaystyle Y= \sum Y_i$ be a snc divisor. We denote by $\omega_{\cC}$ a metric with conic singularities along the $\Q$-divisor 
\begin{equation}\label{pave28}
	\sum_{i\in I}\left(1-\frac{1}{m}\right)Y_i
\end{equation}
where $m$ is a positive integer. By this we mean that locally we have
\begin{equation}\label{pave99}
	\omega_\cC\simeq \sum_{i=1}^r\frac{\sqrt{-1}dz_i\wedge d\ol z_i}{|z_i|^{2-\frac{2}{m}}}
	+ \sum_{i\geq r+1}\sqrt{-1}dz_i\wedge d\ol z_i
\end{equation}
that is to say, $\omega_\cC$ is quasi-isometric with the RHS of \eqref{pave99},
where $z_1\dots z_r=0$ is the local equation of the divisor $Y$.
\smallskip

\noindent Let $(L, h_L)$ be a line bundle endowed with a (singular) metric $h_L$ such that the following holds.
\smallskip

\begin{enumerate}
	\smallskip
	
	\item[(i)] The singularities of the metric $h_L$ are contained in the support of $Y$, i.e. 
	$\displaystyle \varphi_L\simeq \sum q_i\log|z_i|^2$ modulo a smooth function, where $q_i\in \Q$ and as above, $z_i$ represent the local equations of the components of $Y$. Note that the $q_i$ are not necessarily positive. 
	\smallskip
	
	\item[(ii)] For each $q_i\in \Q\setminus \Z$, we assume that $\displaystyle
	\left\lfloor q_i- \frac{1}{m}\right\rfloor= \lfloor q_i\rfloor$ holds true and that $mq_i\in \Z$. 
\end{enumerate}
Note that the requirement $(i)$ implies that the curvature current of $(L, h_L)$ can be written as
	$$i\Theta(L, h_L)= \sum q_i[Y_i]+ \theta_L$$
where $\theta_L$ is a smooth form on $X$.   

\medskip

\noindent Our next result states as follows. 
\begin{thm}\label{CoHo}
	Let $(L, h_L)\to X$ be a line bundle endowed
	with a metric $h_L$ such that the requirements {\rm (i)}-{\rm (ii)} above are satisfied; in addition, we assume $i\Theta(L, h_L) \geq 0$. Let $\omega_\cC$ be a K\"ahler metric with conic singularities as above.
	\begin{enumerate}
		\smallskip
		
		\item[\rm (i)] The following equality holds
		\begin{equation}\label{pave29}\nonumber
			L^2_{n, q}(X_0, L)= {\mathcal H}_{n,q}(X_0, L)\oplus
			\Imm \dbar\oplus \Imm \dbar^\star,
		\end{equation}
		for any $q\geq 0$, where $X_0:= X\setminus Y$.
		\smallskip
		
		\item[\rm (ii)] For any harmonic form $\xi \in {\mathcal H}_{n,q}(X_0, L)$, we have 
		\begin{equation}\label{diffdep}\nonumber
			\dbar \xi =\dbar^\star \xi =(D')^{\star} \xi =0 .
		\end{equation}
		In particular, $\star \xi$ is a $L^2$-holomorpic form on $X$. 
	\end{enumerate}	
\end{thm}  
\medskip

\noindent The proof of Theorem \ref{CoHo} is based on the corresponding version for metrics with Poincar\'e singularities, together with the following result.
A different argument for a more general version of this statement will be given in the next subsection.  We have decided to include this first proof here because the statement \ref{Poinc} below could be useful in other contexts. 
\smallskip

\noindent Let 
\begin{equation}\label{pave35}
	\omega_{Y, \ep}:= \omega_\cC+ \ep \omega_\cP
\end{equation}
be a sequence approximating $\omega_\cC$, where $\omega_\cP$ is a metric with Poincaré singularities on $Y$. We have the following statement.

\begin{thm}\label{Poinc}
There exists a positive constant $C> 0$ (independent of $\ep$) such
	that the following inequality holds
	\begin{equation}\label{pave30}
		\int_{X_0}|u|^2_{\omega_{Y, \ep}}e^{-\varphi_L}dV_{\omega_{Y, \ep}}\leq C \int_{X_0}|\dbar u|^2_{\omega_{Y, \ep}}e^{-\varphi_L}dV_{\omega_{Y, \ep}}
	\end{equation}
	for any $L$-valued form $u$ of type $(p, 0)$ which belongs to the domain of $\dbar$ and which is orthogonal to the
	space of $L^2$ harmonic $(p,0)$--forms (with respect the metric $\omega_{Y, \ep}$).
\end{thm}   

\noindent We first prove Theorem \ref{Poinc}, and then derive Theorem \ref{CoHo} as consequence.

\begin{proof}[Proof of Theorem \ref{Poinc}] The case $\ep= 0$ is quite easy and even if the general case does not follows from it, we give here a direct proof. 
	\smallskip
	
	\noindent We argue by contradiction: if \eqref{pave30} does not holds, then there exists a sequence 
	$(u_k)$ of $(p, 0)$-forms orthogonal to $\Ker(\dbar)$ such that 
	\begin{equation}\label{pave31}
		\int_{X_0} |u_k|^2_{\omega_\cC}e^{-\varphi_L}dV_{\omega_\cC}= 1, \qquad \lim_{k\to\infty}\int_{X_0}|\dbar u_k|^2_{\omega_\cC}e^{-\varphi_L}dV_{\omega_\cC}= 0.
	\end{equation}
	\noindent It follows that the weak limit $u_0$ of $(u_k)_{k\geq 1}$ has the following properties
	\begin{equation}\label{pave32}
		\int_{X_0} | u_0|^2_{\omega_\cC}e^{-\varphi_L}dV_{\omega_\cC}\leq 1, \qquad \dbar u_0= 0.
	\end{equation} 
	Since on the other hand $u_0$ is orthogonal to $\Ker(\dbar)$, it follows that $u_0= 0$. We show next that this contradicts the first equality of \eqref{pave31} for $k$ large enough. 
	
	\noindent Let $\Omega$ be a coordinate ball whose origin belongs to $Y$. The restriction 
	$\displaystyle \dbar u_k|_{\Omega\setminus Y}$ is a closed $(p, 1)$-form, but we can also see it as 
	$(n, 1)$-form with values in the vector bundle $\displaystyle \Lambda^{n-p}T_X\otimes L|_{\Omega}$.
	We equip the tangent bundle $\displaystyle E:= T_X|_{\Omega}$ with the standard metric with conic singularities, and $L$ with $e^{-\phi_L}$ where 
	\begin{equation}\label{pave33}
		\phi_L= \varphi_L+ \Vert z\Vert^2+ C\sum |z_i|^{\frac{2}{m_i}}.
	\end{equation}
	In \eqref{pave33} we denote by $\displaystyle (z_i)_{i=1,\dots, n}$ a coordinate system adapted to $Y$. Then the curvature of $E$ is greater than $\omega_\cC\otimes \Id_E$ and by the $L^2$ estimates applied on $\Omega\setminus Y$ we can solve the equation
	\begin{equation}\label{pave34}
		\dbar v_{\Omega, k}= \dbar u_k, \qquad \int_\Omega |v_{\Omega, k}|^2e^{-\varphi_L}\leq C \int_{X_0}|\dbar u_k|^2_{\omega_\cC}e^{-\varphi_L}dV_{\omega_\cC}.
	\end{equation}
	Here we are using the fact that local term added to the global metric of $L$ in \eqref{pave33} is two-sided bounded.
	\smallskip
	
	\noindent We cover $X$ with a finite subset of coordinate subsets $\Omega$, and we construct the forms $v_{\Omega, k}$ as above (if the set $\Omega$ do not intersects the support of the divisor $Y$, then we just use the flat metric). If all the coefficients $q_i$ in (i) are 
	positive it follows that the holomorphic forms $\displaystyle (u_k- v_{\Omega, k})$ are converging to zero weakly, hence strongly as well. Moreover, the $L^2$ norm of each $v_{\Omega, k}$ tends to zero as $k\to\infty$, so for $k$ large enough the first equality in \eqref{pave31} cannot hold.
	The contradiction we obtained ends the proof of the case $\ep= 0$.
	\smallskip
	
	\noindent Now assume that we only have one negative coefficient $q_1< 0$, then we write $q_1= \lfloor q_1\rfloor+ \{q_1\}$ and replace this coefficient with its fractionally part $\{q_1\}$
	in the expression of the metric. Then can argue as above, by considering $\displaystyle z_1^{- \lfloor q_1\rfloor}(u_k- v_{\Omega, k})$.
	\medskip
	
	\noindent The general case (i.e. for an arbitrary $\ep$) is quite similar, except that the we cannot rely on the "correction" forms $v_{\Omega, k}$ for $\Omega\cap Y\neq \emptyset$. Instead we are using the following two facts. 
	
	The first fact is a version of Proposition~\ref{poinc}: let $U$ be a finite union of coordinate subsets of $X$ containing $Y$. There exists a positive constant $C> 0$  (independent of $\ep$) such that for any $L$-valued $(p,0)$ form $u$ with support in $U$ we have the inequality
	\begin{equation}\label{pave36}
		\frac{1}{C}\int_U|u|^2_{\omega_{Y,\ep}}e^{-\varphi_L}dV_{\omega_{Y,\ep}}\leq 
		\int_{U}|\dbar u|^2_{\omega_{Y,\ep}}e^{-\varphi_L}dV_{\omega_{Y,\ep}}+ \int_{U\setminus \frac{1}{2}U}|u|^2_{\omega_{Y,\ep}}e^{-\varphi_L}dV_{\omega_{Y,\ep}}.
	\end{equation}
	We will not prove \eqref{pave36} here, because it follows directly from the
	arguments invoked for Proposition~\ref{poinc}.
	
	The second fact is that the space of $L^2$ holomorphic
	$p$-forms is independent of $\ep\geq 0$.  
	Let $\Phi$ be a holomorphic $p$--form on $X_0$  with values in $L$. 
	If $\Phi$ is $L^2$ with respect to $(\omega_{Y, \ep}, h_L)$, it is also $L^2$ integrable with respect to $(\omega_{Y}, h_L)$ because of the inequality
	\begin{equation}
		\Vert \Phi\Vert_{\omega_{Y, \ep}}^2\geq \Vert \Phi\Vert_{\omega_{Y}}^2.
	\end{equation}
	We next assume that $\Phi$ is $L^2$ with respect to $(\omega_{Y}, h_L)$. In local coordinates this writes as 
	\begin{equation}
		\int_\Omega\sum|\Phi_I(z)|^2\prod_{j\not\in I} \lambda_j(z) e^{-\varphi_L(z)}d\lambda(z)< \infty
	\end{equation}
	where $\lambda_j$ are the eigenvalues of $\omega_\cC$ with respect to the Euclidean metric on $\big(\Omega, (z_1,\dots, z_n)\big)$. We take the coordinates $z$ so that they are adapted to $(X, Y)$, i.e. $\Omega\cap Y= z_1\dots z_r= 0$.  
	It follows that the holomorphic functions $\Phi_I$ are divisible by
	\begin{equation}\label{pave4411}
		\prod_{i\leq r, i\in I} z_i^{\lfloor q_i\rfloor}
		\prod_{i\leq r, i\not\in I} z_i^{\lfloor 1+ q_i-\frac{1}{m}\rfloor}
	\end{equation}
	cf. notations at the beginning.
	Recall that we have chosen $m$ so that
	\begin{equation}\label{pave422}
		\left\lfloor q_i-\frac{1}{m}\right\rfloor= \left\lfloor q_i\right\rfloor,
	\end{equation}
	for each non-integer coefficient $q_i$ which appears in the expression of the metric $h_L$.
	
	\noindent So summing up we see that by the careful choice of the singularities of the metric $\omega_\cC$ cf. $(ii)$,  for any holomorphic form $\Phi$ satisfying $\Vert \Phi\Vert_{\omega_\cC} ^2< \infty$, we have
	\begin{equation}\label{invarspace}
		\Vert \Phi\Vert_{\omega_{Y, \ep}}^2< \infty \qquad\text{and} \qquad
		\lim_{\ep \rightarrow 0 } \Vert \Phi\Vert_{\omega_{Y, \ep}}^2 = \Vert \Phi\Vert_{\omega_\cC}^2 .
	\end{equation}
	
	The rest of the proof is absolutely similar (i.e. by contradiction) to the case
	of Theorem \ref{prop:2}: we suppose by contradiction that there exists a sequence $u_\ep$ such that $\|u_\ep\|_{\omega_{Y,\ep}}=1$, $u_\ep$ is orthogonal to the space of $L^2$-holomorphic forms with respect to $\omega_{Y, \ep}$ and 
	$\lim_{\ep \rightarrow 0} \|\dbar u_\ep\|_{\omega_{Y,\ep}} =0$. By passing to some subsequence, $u_\ep$ converges weakly to some $(p,0)$-form $u$ on $X_0$. Then $\|u\|_{\omega_{\mathcal{C}}} \leq 1$ and $\dbar u =0$. Therefore $u$ is a holomorphic $L^2$-form on $X$. Since $u_\ep$ is orthogonal to the holomorphic forms, we have $\int_X \langle u, u_\ep\rangle_{\omega_{Y,\ep}} =0$. Together with \eqref{invarspace}, we obtain $u =0$. Then we obtain a contradiction by the same arguments as in Theorem \ref{prop:2} combined with \eqref{pave36}.
\end{proof}
\medskip

\noindent We prove next Theorem \ref{CoHo}.

\begin{proof} Let $\xi$ be any $L^2$-form of type $(n,q)$ with values in $(L, h_L)$ (with respect to the metric $\omega_{\cC}$). Then we also have 
	$\displaystyle \Vert\xi\Vert_{\omega_{Y, \ep}}< \infty$ and thus we write
	\begin{equation}\label{pave37}
		\xi= \xi_\ep+ \dbar(\eta_\ep)+ \dbar^\star(\tau_\ep),
	\end{equation}
	where $\xi_\ep$ is harmonic, $\eta_\ep$ is orthogonal to $\Ker(\dbar)$ and $\tau_\ep$ orthogonal on $\Ker(\dbar^\star)$. We have 
	\begin{equation}\label{pave38}
		\Vert \xi_\ep\Vert^2 _{\omega_{Y, \ep}}\leq \Vert \xi\Vert^2 _{\omega_{Y, \ep}}\leq \Vert \xi\Vert^2 _{\omega_{Y}}
	\end{equation}
	and therefore the limit of $(\xi_\ep)_{\ep> 0}$ will be harmonic with respect to
	$\omega_\cC, h_L$. Moreover, thanks to Theorem \ref{Poinc}, we have 
	\begin{equation}\label{pave39}
		\Vert \eta_\ep\Vert^2 _{\omega_{Y, \ep}}\leq C \Vert \dbar \eta_\ep\Vert^2 _{\omega_{Y, \ep}}\leq 
		C \Vert \xi\Vert^2 _{\omega_{Y}}
	\end{equation}
	and thus we can extract a limit of $(\eta_\ep)_{\ep> 0}$ on compacts of $X_0$.
	
	\bigskip
	
	\noindent In order to deal with the family $(\tau_\ep)_{\ep> 0}$, thanks to the uniform estimate of Theorem \ref{Poinc}, the proof of Corollary \ref{usecor} applies. We get thus a constant $C> 0$ so that 
	\begin{equation}\label{pave43}
		\Vert \tau_\ep\Vert^2_{\omega_{Y,\ep}}\leq C\Vert \dbar^\star(\tau_\ep)\Vert^2_{\omega_{Y,\ep}}
	\end{equation}
	and the RHS of \eqref{pave43} is bounded uniformly with respect to $\ep$. We can therefore take the limit $\ep\to 0$ and we get $\xi_0$ harmonic, $\eta_0\in \Dom(\dbar)$ and $\tau_0\in \Dom(\dbar^\star)$ such that
	\begin{equation}\label{pave44}
		\xi= \xi_0+ \dbar\eta_0+ \dbar^\star\!(\tau_0) .
	\end{equation}
	
	\medskip
	
	To prove that $\mathcal{H}_{n,q} (X_0, L), \Imm \dbar$ and $\Imm \dbar ^*$ are orthogonal with each other, we use the partition of unity and the local covering. For example, let $\xi_0 \in \mathcal{H}_{n,q} (X_0, L)$ and $\eta_0 \in L^2$. Let $\theta_i$ be a partition of unity. We have
	$$\int_{X_0} \langle  \xi_0, \dbar \eta_0 \rangle =\sum_i \int_{X_0} \langle  \xi_0,  \dbar (\theta_i\eta_0) \rangle .$$
	Let $\pi_i : V_i \rightarrow U_i$ be a local ramified cover such that $\pi_i ^* \omega_\cC$ is smooth and
	$Supp \theta_i \Subset U_i$. Then 
	$$\int_{X_0} \langle  \xi_0,  \dbar (\theta_i\eta_0) \rangle  = 
	\int_{U_i } \langle  \xi_0,  \dbar (\theta_i\eta_0) \rangle = \frac{1}{r} \int_{V_i }
	\langle  \pi_i ^*\xi_0,  \dbar (\pi_i ^* (\theta_i\eta_0) ) \rangle_{\pi_i ^* \omega_\cC} =0 ,$$
	where the last equality comes from the fact that $\pi_i ^* \omega_\cC$ is smooth on $V_i$ and 
	$Supp \theta_i \Subset V_i$.
	
	\medskip
	
	To conclude the proof, it remains to check \eqref{diffdep}. Let $\xi \in \mathcal{H}_{n,q} (X_0, L)$ (with respect to the conic metric). Then $\xi$ is still $L^2$ with respect to $\omega_{Y,\ep}$ for every $\ep$. We consider the Hodge decomposition of $\xi$ with respect to $\omega_{Y,\ep}$: 
	$$\xi = \xi_{\ep} +\dbar \eta_\ep + \dbar^{\star_\ep}\tau_\ep .$$
	Then $\xi_{\ep}$ is harmonic with respect to the complete metric $\omega_{Y,\ep}$.
	In particular, by using Bochner, we know that 
	\begin{equation}\label{diff}
		\dbar \xi_{\ep} = \dbar^{\star_\ep} \xi_{\ep} = (D')^{\star_\ep} \xi_{\ep} =0 .
	\end{equation}
	By passing to some subsequence, the above argument implies that $\xi_{\ep} , \eta_\ep$ and $\tau_\ep$ converge to some $\xi_0, \eta_0$ and $\tau_0$ which are $L^2$ with respect to $\omega_\cC$ and we have 
	$$\xi = \xi_0 + \dbar  \eta_0 + \dbar^\star \tau_0 .$$
	As we proved that $\mathcal{H}_{n,q} (X_0, L), \Imm \dbar$ and $\Imm \dbar ^*$ are orthogonal with each other, we obtain that $\xi =\xi_0$. Together with \eqref{diff}, \eqref{diffdep} is proved.
\end{proof}

\begin{remark}{\rm The decomposition 
		$$L^2 _{n,q} = \Ker \Delta'' \oplus \Imm \Delta'' $$
		with respect to the Poincaré type/conic metrics, still holds, as we will see next. Given our previous results, all we need to prove is the image of $\Delta''$
		is closed.
		
		\noindent We prove it for the Poincaré case, and the conic case is similar.
		If (by contradiction) this is not true, then there exists a sequence of smooth $f_i$ of compact support such that
		\begin{equation}\label{van11}
			\|f_i\|^2 _{L^2} =1, \qquad \|\Delta'' f_i\|^2 _{L^2}  \rightarrow 0.
		\end{equation}
		Moreover, each of the forms $f_i$ is orthogonal to the space of harmonic forms, i.e. to $\Ker (\dbar)$ and $\Ker(\dbar^\star)$.
		\smallskip
		
		\noindent As consequence of \eqref{van11} we have
		\begin{equation}\label{van12}
			\langle f_i, \Delta'' f_i \rangle \rightarrow 0
		\end{equation}
		as $i\to\infty$. On the other hand, we have $\langle f_i, \Delta'' f_i \rangle= \|\dbar f_i\|^2 + \|\dbar^* f_i\|^2$, and this leads to a contradiction as follows. We can write
		\begin{equation}\label{van13}
			f_i= \dbar \alpha_i+ \dbar^\star\beta_i,
		\end{equation}
		according to Theorem \ref{Hodge}. Moreover, it follows from the proof of  Theorem \ref{Hodge} that
		\begin{equation}\label{van14}
			\|f_i- \dbar \alpha_i\|^2\leq C\|\dbar f_i\|^2, \qquad  \|f_i- \dbar^\star \beta_i\|^2\leq C\|\dbar^\star f_i\|^2
		\end{equation}
		since $f_i- \dbar \alpha_i$ is orthogonal to $\Ker(\dbar)$ and $f_i- \dbar^\star \beta_i$ is orthogonal to $\Ker(\dbar^\star)$. Adding the two inequalities shows that
		\begin{equation}\label{van15}
			\|f_i\|^2\leq C(\|\dbar f_i\|^2+ \|\dbar^\star f_i\|^2),
		\end{equation}
		which is impossible as $i\to\infty$.}
\end{remark}
\subsection{Hodge decomposition for metrics with conic singularities, II}
The setting here will be the same as in the previous subsection \ref{subsectconic}. 
We will use systematically the ramified covers corresponding  to the orbifold structure
\begin{equation}\label{orbi1}
	\Delta:= \sum_{i\in I}\left(1-\frac{1}{m}\right)Y_i
\end{equation}
in order to establish a generalization of Theorem \ref{CoHo}. 

\medskip

\noindent Let $\displaystyle (V_i, z_i)_{i\in I}$ be a finite cover of $X$ with coordinate charts such that
\begin{equation}\label{deck2}
	z_i^1\dots z_i^r= 0
\end{equation}
is the local equation of the divisor $Y$ when restricted to the
set $V_i$. We then consider the local ramified maps
\begin{equation}\label{deck3}
	\pi_i: U_i\to V_i, \qquad \pi_i(w_i^1,\dots, w_i^n):=
	\left((w_i^1)^{m},\dots, (w_i^r)^{m}, w_i^{r+1},\dots w_i^n\right) 
\end{equation}
which define the orbifold structure corresponding to $(X, \Delta)$.
\medskip

\noindent We first define the analog of smooth forms in our geometric context.

\begin{defn}\label{forms}
	Let $\phi$ be a $\cC^\infty$ form of $(p, q)$--type with values in $L$ defined on $X_0:= X\setminus Y$. We say that $\phi$ is smooth in conic
	sense if the quotient of the local inverse images 
	\begin{equation}\label{deck4}
		\phi_i:= \frac{1}{ w_i^{qm}}\pi_i^\star(\phi|_{V_i})
	\end{equation}  
	admit a $\cC^\infty$-extension to $U_i$. In order to simplify the writing, in the equality above we are using the notation
	$w_i^{qm}:= (w_i^1)^{q_1 m} \cdots (w_i^r)^{q_r m }$.    
\end{defn}
\medskip

\noindent In the absence of the twisting bundle $L$, this corresponds to the usual definition of ``orbifold differential forms''. Here in some sense the idea is the same, except that we take into account the singularities of the metric $h_L$ of $L$. In what follows we will establish a few simple
properties.

\begin{proposition}
	Let $\phi$ be a smooth $(n,q)$-form on $X$ with value in $L$ in the standard sense, and we suppose that $q_i< 1$ for every $i$. Then $\phi$ is smooth in conic sense.
\end{proposition}
\begin{proof}
	The matter is immediate, via a local calculation. This is based on the fact that
	$\displaystyle \frac{d(z^m)}{z^{qm}}$
	is smooth (recall that $qm$ is an integer).  
\end{proof}  

\begin{remark}{\rm 
		Let $\phi$ be a smooth $L$-valued $(p,q)$-form on $X$. Even if $q_i \in ]0,1[$ for every $i$, it is possible that $\phi$ is not smooth in conic sense.  But we will show later that it defines a current in conic sense.}
\end{remark}
\medskip

\noindent The following statement establishes a correspondence between the
intrinsic differential operators associated to the data $(X, \omega_\cC)$ and $(L, h_L)$ and the local ones.
\begin{proposition}\label{orbderiv}
	Let $\phi$ be a  $L$-valued $(p,q)$-form, smooth in orbifold sense.  
	Then its "natural'' derivatives $D' _{h_L}, D_{h_L} ^\star, \cdots $ are also smooth in the orbifold sense. Moreover, we have:
\begin{enumerate}
\smallskip
		
\item[\rm (1)] $\sup_{X\setminus Y}|\phi|_{h_L, \omega_\cC}< \infty$, i.e. forms which are smooth in conic sense are bounded.       

\smallskip
		
\item[\rm (2)] The following equalities hold true
$$
\pi_i ^{\star}(D' \phi) = w_i^{qm} D' \phi_i ; \qquad \pi_i ^\star(\dbar \phi) = w_i^{qm} \dbar \phi_i,$$
as well as 
$$ \pi_i ^{\star}({D'}^{\star} \phi) = w_i^{qm} {D'}^\star \phi_i ; \qquad   
\pi_i ^{\star}(\dbar^{\star} \phi) = w_i^{qm} \dbar^{\star} \phi_i$$	
where the $D'$ on the LHS is $\displaystyle D'_{h_L}$ and the rest of the notations are $D'\psi:= \partial \psi-\partial (\varphi_{L, 0}\circ\pi_i)\wedge \psi$ and $\varphi_L = \sum q_i  \log |z_i|^2+ \varphi_{L, 0}$.
		
		\item[\rm (3)] Let $\Delta'': =[\dbar, 	\dbar^{\star}]$ be the Laplace with respect to $(\omega_\cC, h_L)$. Then we have 
		$$
		\pi_i ^{\star}	\Delta'' \phi = w_i^{qm} \cdot \Delta''_{sm} \phi_i,$$
		where $\Delta''_{sm} $ is the Laplace for the local, non-singular setting $(\pi_i ^\star \omega_\cC, \varphi_{L, 0}\circ\pi_i)$
	\end{enumerate}
\end{proposition}
\begin{proof} The first inequality is clear; for the rest we will only discuss the first equality of (2), since the verification of all the others is completely analogue. Also, we will drop the index $i$ in order to simplify the notations. By definition, locally on $V$  we have
	\begin{equation}
		D' _{h_L} \phi= \partial \phi- \Big(\partial \varphi_{L, 0}+ \sum q_\alpha \frac{dz_\alpha}{z_\alpha}\Big)\wedge \phi
	\end{equation}
	and its pull-back to $U$ via the map $\pi$ is equal to
	\begin{equation}
		\pi^\star(D' _{h_L} \phi)= \partial \pi^\star \phi- \Big(\partial \varphi_{L, 0}\circ\pi+ m\sum q_\alpha \frac{dw_\alpha}{w_\alpha}\Big)\wedge \pi^\star\phi.
	\end{equation}
	The equality
	\begin{equation}
		\frac{1}{w^{mq}}\pi^\star(D' _{h_L} \phi)= \partial \frac{\pi^\star \phi}{w^{mq}}- \partial (\varphi_{L, 0}\circ\pi)\wedge \frac{\pi^\star\phi}{w^{mq}}
	\end{equation}
	follows immediately and the RHS equals precisely $\displaystyle D'\Big(\frac{1}{w^{qm}}\pi^{\star}\phi\Big)$.
\end{proof}

\begin{exemple}{\rm
The terminology we are using could be sometimes misleading. Let $L$ be the trivial bundle 
over $X=\mathbb D$, endowed with the metric $\varphi_L = - \frac{1}{2}\ln |z|$.
Then the form $\displaystyle \frac{dz}{z}$ defined over the pointed disk is smooth in the sense of our Definition \ref{forms}. Moreover, its $\dbar$ is also smooth in 
conic sense: the requirement is that  $\displaystyle w\dbar\Big(\frac{dw}{w}|_{\mathbb D^\star}\Big)$ 
extends over the disk and this is obviously the case. 
}\end{exemple}	
\medskip

\noindent A first consequence of the statement \ref{orbderiv} is that forms which are smooth in conic sense behave well by integration by parts.
\begin{cor}\label{int}
  Let $\alpha$ and $\beta$ be an $L$-valued $(p,q)$-form and an $L^\star$-valued $(n-p-1, n-q)$ form, respectively, both smooth in conic sense.
Then the usual integration by parts formula holds true
\begin{equation}\label{van40}
\int_XD'_{h_L}\alpha\wedge\beta=
(-1)^{p+q+1}\int_X\alpha\wedge D'_{h_L^\star}\beta.
\end{equation}  
\end{cor}
\begin{proof}
  Let $(\mu_\ep)_{\ep>0}$ be the family of truncation functions corresponding to $Y$. We obviously have
\begin{equation}\label{van41}
\int_XD'_{h_L}(\mu_\ep\alpha)\wedge\beta=
(-1)^{p+q+1}\int_X\mu_\ep\alpha\wedge D'_{h_L^\star}\beta
\end{equation}
for each positive $\ep$. The integrals $\displaystyle \int_X\mu_\ep\alpha\wedge D'_{h_L^\star}\beta$ and $\displaystyle \int_X\mu_\ep D'_{h_L}\alpha\wedge\beta$
are converging to the expected value by Lebesgue theorem, since
\begin{equation}\label{van42}
\sup_{X\setminus Y}(|D'_{h_L}\alpha|+|\alpha|+ |D'_{h_L^\star}\beta|+|\beta|)< \infty
\end{equation}
(cf. Proposition \ref{orbderiv} above). The remaining term is bounded by
\begin{equation}\label{van43}
\int_X|\partial \mu_\ep| dV_\cC
\end{equation}
up to a uniform constant --because of \eqref{van42}-- and this converges to zero
as $\ep\to 0$.
\end{proof}

\medskip

\noindent The following corollary is very useful.
\begin{cor}\label{smoothn}
	Let $\phi$ be a $L$-valued $L^2$-form on $X$ such that $\Delta''(\phi) =\psi$ holds in the sense of currents on $X$ for some form $\psi$ smooth in conic sense. Then $\phi$ is smooth in conic sense.
\end{cor}	

\begin{proof}
	Let $V\subset X$ be a coordinate open subset, such that we have    
	a local ramified cover $\pi: U \rightarrow V$ given by the orbifold structure.
	By hypothesis we have 
	\begin{equation}\label{int12}
		\int_V\phi\wedge \ol {\Delta''\tau} e^{-\varphi_L}= \int_V\psi\wedge \ol {\tau} e^{-\varphi_L}
	\end{equation}
	for any $L$-valued form $\tau$ with compact support in $V$, smooth in orbifold sense.
	The Laplace operator in \eqref{int12} is induced by $\omega_{\cC}$ and $h_L$.
	
	\medskip
	
	We first prove that
	\begin{equation}\label{int15}
		\int_U\frac{\pi^\star(\phi)}{w^{mq}}\wedge \ol {\Delta''_{sm}(\rho)} = \int_U\frac{\pi^\star(\psi)}{w^{mq}}\wedge \ol {\rho}	
	\end{equation}
	for any smooth form $\rho$ of compact support in $U$. 
	
	In fact, the $m$-ramified cover $\pi: U \rightarrow V$ induces the Galois group $\gamma_1, \cdots , \gamma_m$ which are the automorphisms  $\gamma_i: U \rightarrow U$ invariant over $V$.
	Since $\pi^* \psi$ is $\gamma_i$-invariant, we have 
	$$\int_U\frac{\pi^\star(\psi)}{w^{mq}}\wedge \ol {\rho}	 =\frac{1}{m} \sum_i
	\int_U \pi^\star(\psi) \wedge \gamma_i ^* \frac{\ol {\rho}}{w^{mq}} 
	=\frac{1}{m} \sum_i
	\int_U \pi^\star(\psi) \wedge \gamma_i ^* (\ol {w^{mq} \cdot \rho}) \cdot e^{-\pi^*\varphi_L}	$$
	$$=\frac{1}{m} \int_U \pi^\star(\psi) \wedge \ol G \cdot e^{-\pi^* \varphi_L} ,$$
	where $G:= \sum_i  \gamma_i ^* (w^{mq} \rho)$ is $\pi$-invariant and the quotient $\frac{G}{w^{mq}}$ is smooth by construction. By definition, $G$ is the pull back of some form $\tau$ on $V$, smooth in conic sense.
	Therefore we have
	\begin{equation}\label{addd1}
		\int_U\frac{\pi^\star(\psi)}{w^{mq}}\wedge \ol {\rho}	 = \int_V \psi \wedge \ol \tau \cdot e^{-\varphi_L} .
	\end{equation}
	
	For the term $	\int_U\frac{\pi^\star(\phi)}{w^{mq}}\wedge \ol {\Delta''_{sm}(\rho)}$, we do the same argument:
	$$	\int_U\frac{\pi^\star(\phi)}{w^{mq}}\wedge \ol {\Delta''_{sm}(\rho)} =\frac{1}{m} \sum_i
	\int_U \pi^\star(\phi) \wedge \gamma_i ^* \frac{\ol {\Delta''_{sm}(\rho)} }{w^{mq}} 
	=\frac{1}{m} \sum_i
	\int_U \pi^\star(\phi) \wedge \gamma_i ^* (\ol {w^{mq} \cdot \Delta''_{sm}(\rho)}) \cdot e^{-\pi^*\varphi_L}	$$
	$$=\frac{1}{m} \int_U \pi^\star(\phi) \wedge \ol H \cdot e^{-\pi^* \varphi_L} ,$$
	where $H := \sum_i  \gamma_i ^* (w^{mq}\Delta''_{sm}(\rho) )$ is $\pi$-invariant and the quotient $\frac{H}{w^{mq}}$ is smooth. Then there exists a form $\theta$ on $V$ smooth in conic sense such that $H = \pi^* \theta$. 
	Then 
	\begin{equation}\label{addd2}
		\int_U\frac{\pi^\star(\phi)}{w^{mq}}\wedge \ol {\Delta''_{sm}(\rho)}	 = \int_V \phi \wedge \ol \theta \cdot e^{-\varphi_L} .
	\end{equation}
	
	Since $\Delta^{''} _{sm} $ commute with $\gamma_i$, we obtain $H = w^{qm} \Delta^{''} _{sm} (\frac{G}{ w^{qm}})$. Together with Proposition \ref{orbderiv}, 
	we know that $\theta = \Delta^{''} \tau$ on $V$. Combining this with \eqref{addd1}, \eqref{addd2} and \eqref{int12}, we obtain finally
	$$	\int_U\frac{\pi^\star(\phi)}{w^{mq}}\wedge \ol {\Delta''_{sm}(\rho)} = 
	\int_U\frac{\pi^\star(\psi)}{w^{mq}}\wedge \ol {\rho} . $$
	\medskip
	
	\noindent Therefore, the equality
	\begin{equation}\label{int16}
		\Delta''_{sm}(\wh \phi)= \frac{\pi^\star(\psi)}{w^mq}  
	\end{equation}
	holds in the sense of distributions on $U$, where $\displaystyle \wh \phi:= \frac{\pi^\star(\phi)}{w^{mq}}$ is moreover $L^2$. The corollary follows. 
\end{proof}	
\medskip

\noindent Another consequence is the following version of Bochner formula.
\begin{proposition}\label{bochner}
	Let $\phi$ be a $L$-valued $(p,q)$-form, smooth in orbifold sense.
	Then the equality
	\begin{equation}\label{bochnerconic}
		\Delta'' \phi = \Delta' \phi + [\theta_L, \omega_\cC] \phi 
	\end{equation}
	holds pointwise on $X\setminus Y$. Moreover the following Bochner formula holds.
	\begin{align}\nonumber
		\int_X|\dbar \phi|^2e^{-\varphi_L}dV_{\cC}+ \int_X|\dbar^\star \phi|^2e^{-\varphi_L}dV_{\cC}= & \int_X|D'\phi|^2e^{-\varphi_L}dV_{\cC}+ \int_X|D'^{\star} \phi|^2e^{-\varphi_L}dV_{\cC}\cr + & \int_X\langle[\theta_L, \omega_\cC]\phi, \phi\rangle e^{-\varphi_L}dV_{\cC}\cr
	\end{align}
\end{proposition}
\begin{proof}
	On every local ramified cover $\pi_i : U_i \rightarrow V_i$, thanks to Proposition \ref{orbderiv}, we have 
	$$	\pi_i ^{\star}(\Delta'' \phi) = w_i^{qm} \cdot \Delta''_{sm} \phi_i .$$
	Using the same type of arguments we infer that we have
	$$	\pi_i ^{\star}(\Delta' \phi) = w_i^{qm} \cdot \Delta'_{sm} \phi_i .$$
	Together with the usual Bochner equality $\displaystyle \Delta''_{sm}  = \Delta'_{sm} + [i\Theta_{\pi^\star\varphi_0}, \pi_i ^\star \omega_\cC]$, we obtain \eqref{bochnerconic}. 
	\smallskip
	
	\noindent In order to obtain the integral identity, we first see that the formula holds if we replace $\phi$ with $\mu_\ep \phi$, where the cut-off function $\mu_\ep$ is defined in Lemma \ref{cutoff}. 
	Since the error terms e.g. of type
	\begin{equation}\label{van16}
		\int_X\langle \dbar\mu_\ep\wedge \phi, \dbar\phi\rangle e^{-\varphi_L}dV_\cC
	\end{equation}
	tend to zero, we are done. Note that here we are using the fact that the function $|\phi| |\dbar\phi|e^{-\varphi_L}$ is bounded (and not only $L^1$).
\end{proof}
\medskip

\noindent Yet another corollary states as follows.
\begin{cor}\label{kompsup} Let $\tau$ be a smooth $(n, q)$-form with values in $(L, h_L)$ and support in $X\setminus Y$. 
We consider the (usual) orthogonal decomposition
\begin{equation}\label{van31}
	\tau= \tau_1+ \tau_2
\end{equation}
where $\tau_1\in \Ker \dbar$ and $\tau_2\in (\Ker \dbar)^\perp$. \begin{itemize}

\item We have 
\begin{equation}\label{van32}\nonumber
	\dbar \tau_1= 0, \qquad  \dbar ^\star \tau_1= 0
\end{equation}
when restricted to $U$, the complement of the support of $\tau$. In other words, $\tau_1|_U$ is harmonic, hence smooth in conic sense.
\smallskip

\item As a consequence of the first bullet, Bochner formula holds true for the projection $\tau_1$ of $\tau$ on the kernel of 
$\dbar$. The is true more generally for forms $\rho$ such that $\displaystyle \sup_{X\setminus Y}|\rho|_{h_L, \omega_\cC}< \infty$.
\end{itemize}
\end{cor}
\medskip

\begin{cor} If $i\Theta_{h_L} \geq 0$ (resp. $i\Theta_{h_L} \leq 0$), and $\phi$ is a $L$-valued $\Delta''$-harmonic $(n,q)$-form (resp. $(0,q)$-form) smooth in conic sense.
Then $D_{h_L} \phi = D_{h_L}^\star \phi =0$ and $\theta_L \wedge \phi =0$.
\end{cor}	

\begin{proof}
	All we have to do is to combine the previous results: harmonic forms are smooth in conic sense, hence Bochner formula applies and the corollary follows.
\end{proof}

\medskip

\noindent Like in the smooth case, we can define the Hodge operators $\star$ and $\sharp= \ol \star $ in our setting, namely give a $L$-value $(p,q)$-form $s$,  there exists a unique $L^*$-value $(n-p,n-q)$-form,  denoted by $\sharp t$, such that for every $L$-valued $(p,q)$-form $s$, we have
$$\langle s, t\rangle_{h_L} dV_{\omega_\cC} = s\wedge \sharp t ,$$
where $s\wedge \sharp t $ is calculated by using the natural pairing $L\otimes L^\star \rightarrow \mathbb C$. We have the following statements, whose proofs are left to the interested readers.

\begin{proposition}
	Let $t$ be a $L$-valued $(p,q)$-form, smooth in the conic sense. Then $\sharp t$ is a $L^\star$-valued $(n-p, n-q)$-form, smooth in the conic sense (with respect to $ (\omega_\cC , h_{L}^\star)$).
\end{proposition}

\begin{proposition}\label{sharppair}
	Let $t$ be a $\Delta''$-harmonic form with value in $L$. Then $\sharp t$ is $\Delta''$-harmonic form with value in $L^*$ 
\end{proposition}

\medskip

\noindent The usual G{\aa}rding and Sobolev inequalities together with Rellich embedding theorem still hold in conic setting. They are established by their usual versions via the local uniformisations of $(X, \Delta)$. We will state and comment G\aa rding inequality, but prior to doing that, we introduce the analogue of Sobolev spaces in our context.
\smallskip

\begin{defn} We denote by $W^k _{p,q} (X, L)$ the space of $L^2$ forms $u$ of type $(p,q)$ with values in $L$, such that its pull back 
	$\displaystyle \frac{\pi_i ^\star u}{w_i ^{qm}}$ via the ramified coverings belongs to the usual Sobolev space $W^k(U_i)$. The Sobolev norm of $W^k _{p,q} (X, L)$ is obtained by a partition of unit and the usual Sobolev norm of the local forms $\displaystyle \frac{\pi_i ^\star u}{w_i ^{qm}}$.
\end{defn}

\noindent Given a form $u \in W^k _{p,q} (X, L)$, we can approximated it by smooth forms $u_\ep$ in the sense of conic.  In fact, by using the partition of unity, it is sufficient to consider the case when $u$ is supported in some open set $V_i$. Let $z$ be the coordinate on $U_i$. Let $\rho_\ep (\|w\|^2)$ be the standard convolution function. Let $G$ be the Galois group of the covering $\pi: U_i \rightarrow V_i$ and let $\gamma\in G$. Then $\rho_\ep (|w|^2)$ is $G$-invariant. Assume for simplicity that $p=q=0$; we then consider the function
$$ v_\ep (w) := \frac{(w^{-qm} \cdot \pi^\star u) \star \rho_\ep}{w^{-qm}}.$$
Note that $\gamma (w)^{-qm} = c_\gamma \cdot w^{-qm}$ for some constant $c_\gamma$ independent of $w$. Then 
$$v_\ep (\gamma (w))  = \frac{  \int_{U_i}  (\gamma (w)-\tau)^{-qm} \cdot \pi^\star u (\gamma (w)-\tau) \cdot  \rho_\ep (\|\tau\|^2) dV_\tau }{\gamma (w)^{-qm}} $$
$$= \frac{ c_\gamma \int_{U_i}  (w-\gamma^{-1}(\tau))^{-qm} \cdot \pi^\star u ( w -\gamma^{-1}(\tau)) \cdot  \rho_\ep (\|\tau\|^2) dV_{\tau} }{c_\gamma \cdot w^{-qm}}$$
$$ =\frac{ c_\gamma \int_{U_i}  (w- \tau)^{-qm} \cdot \pi^\star u ( w -\tau) \cdot  \rho_\ep (\|\gamma (\tau)\|^2) dV_{\gamma(\tau)} }{c_\gamma \cdot w^{-qm}} =v_\ep (w) ,$$
where we use the fact that $\pi^* u$, $\rho_\ep (\|w\|^2)$ and $dV_w$ are $G$-invariant. 
Therefore $v_\ep$ is $G$-invariant and we can find some form $u_\ep$ on $U_i$ such that $v_\ep = \pi^* u_\ep$.  Then $u_\ep$ is smooth in conic sense and $\|u_\ep - u\|_{W^k_{p,q}} \rightarrow 0$.
\smallskip

\noindent The general case of an arbitrary $(p, q)$--form is similar, obtained by regularising locally its coefficients.
\medskip

\noindent We have the following version of G{\aa}rding theorem.
\begin{thm}
	Let $\Delta'' := [\dbar^\star, \dbar]$ be the Laplace operator induced by $(h_L, \omega_\cC)$. 	
	Then for any $u\in L^2 _{p,q} (X, L)$, if $\Delta''u \in W^k _{p,q} (X, L)$, then $u \in W^{k+2} _{p,q} (X, L)$ and 
	$$ \|u\|_{k+2} \leq C_k ( \|\Delta''u \|_k + \|u\|_0 ).$$
\end{thm}	

\begin{proof}
	We first observe that this is a direct consequence of the usual G\aa rding inequality if the form
	$u$ has compact support contained in one of the sets $V_i$. Moreover, the ``dictionary" established in Proposition \ref{orbderiv}, 
	shows that one could define the space $\big(W^k _{p,q} (X, L), \|\cdot \|_k\big)$ directly on $X$, by using the covariant derivative on $L$-valued forms induced by $h_L$ and $\omega_\cC$. The general case is obtained by a partition of unit argument.   
\end{proof}	
\medskip

\noindent As a consequence, the Hodge decomposition still holds in the conic setting, as we see e.g. from the proof presented in \cite{bookJP},
Chapter 6.
\begin{thm}\label{CoHo1}
	Let $(L, h_L)\to X$ be a line bundle endowed
	with a metric $h_L$ such that the requirements {\rm (i)}-{\rm (ii)} at the beginning of this section are satisfied. 
	Let $\omega_\cC$ be a K\"ahler metric with conic singularities as in \eqref{pave28}.
	Then we have the following Hodge decomposition
	\begin{equation}\label{pave29111}
		L^2 _{p, q}(X, L)=\Ker \Delta''_{h_L} \oplus
		\Imm \dbar\oplus \Imm \dbar^\star
	\end{equation}
	and
	\begin{equation}\label{pave291}
		\cC^\infty _{p, q}\big(X, (L, h_L)\big)= \Ker \Delta'' _{h_L} \oplus
		\Imm \Delta'' _{h_L} ,
	\end{equation}
	where $\displaystyle \cC^\infty _{p, q}\big(X, (L, h_L)\big)$ and $L^2 _{p, q}(X, L)$ are the spaces of $L$-valued $(p,q)$-forms smooth in conic sense and $L^2$, respectively.
\end{thm}  	

\noindent We will not give a formal argument for this statement, because it follows as in the non-singular case (by using the ellipticity of 
the Laplace operator combined with the three big theorems mentioned above). Similarly, existence of a positive constant $C> 0$ such that the familiar elliptic estimate 
\begin{equation}\label{van36}
	\int_X|u|^2e^{-\varphi_L}dV_\cC\leq C \int_X|\Delta'' _{h_L} u|^2e^{-\varphi_L}dV_\cC
\end{equation}
holds true for all $L^2$-forms orthogonal to $\Ker \Delta'' _{h_L}$ follows as consequence of the corresponding statement in non-singular setting.
\medskip

\noindent To end this section, we collect a few results concerning the $L^2$ theory for the data $(L, h_L)$ and $(X, \omega_\cC)$. The first and foremost of them is the following statement, due to J.-P. Demailly.
\begin{thm}\cite{bookJP}\label{JP} Let $v\in L^2_{n, q}(X, L)$ be an $L$-valued $(n, q)$-form such that $\dbar v= 0$. We assume that $i\Theta(L, h_L)\geq 0$ and that the coefficients $q_i\in ]0, 1[$. Moreover, suppose that the integrability condition 
	$$\int_X\langle[\theta_L, \Lambda_{\omega_\cC}]^{-1}v, v\rangle e^{-\varphi_L}dV_\cC< \infty$$
	holds true. Then there exists $u\in L^2_{n, q-1}(X, L)$ such that 
	\begin{equation}
		\dbar u= v, \qquad \int_X|u|^2e^{-\varphi_L}dV_\cC\leq \int_X\langle[\theta_L, \Lambda_{\omega_\cC}]^{-1}v, v\rangle e^{-\varphi_L}dV_\cC.
	\end{equation}
\end{thm}
This is done by an approximation process (by reducing to the complete case) and it holds in much more general setting, see \cite{Dem12}.
\medskip

\noindent In our specific situation, an important observation is that is would be sufficient to solve the equation
\begin{equation}\label{van28}
	\dbar u= v
\end{equation}
in the sense of distributions on $X\setminus Y$, provided that $u$ is in $L^2$. Again, this can be seen as in the usual, non-singular setting. Indeed, let $\phi$ be any smooth $(0, n-q)$ form with values in the dual bundle $L^\star$ of the bundle $L$. We have 
\begin{equation}\label{van29}
	\int_X\dbar u\wedge (\mu_\ep\phi)= \int_X v\wedge (\mu_\ep\phi)
\end{equation}
so in order to show that \eqref{van28} holds globally on $X$ it would be enough to prove that we have 
\begin{equation}\label{van30}
	\lim_\ep \int_Xu\wedge \phi \wedge \dbar \mu_{\ep}=0.
\end{equation}
This is however clear, by Cauchy-Schwarz: we have $\displaystyle \sup_{X\setminus Y}|\phi|^2_{\omega_\cC}e^{\varphi_L}< \infty$ since 
the coefficients $q_i$ corresponding to the singularities of $h_L$ are positive and $\phi$ is smooth on $X$. Thus \eqref{van30} boils down to showing that the limit of $\displaystyle \int_X|\dbar \mu_{\ep}|^2_{\omega_\cC}dV_\cC$ is equal to zero. But this is clear by a quick computation that we skip.

\bigskip

\noindent We introduce the following notion, which -so to speak- represents the formalisation of the examples that 
we discuss afterwards.  

\begin{defn}\label{currents}
 A $(p, q)$-current $T$ of with values in $(L, h_L)$ is simply a $L$-valued
 current such that there exist a constant $C> 0$ and a positive integer $d$ such that the inequality
	\begin{equation}\label{deck5}
		\left|\int_{X}T\wedge \ol {\star \phi}e^{-\varphi_L} \right|^2\leq C \sum_{s=0}^d \sup_{X\setminus Y}|\nabla ^s\phi|^2_{h_L, \omega_\cC}   
	\end{equation}
holds for all $(p, q)$-forms $\phi$, with values in $L$ which moreover are smooth in conic sense. Here $h_0$ is a non-singular metric on $L$.
\end{defn}
\smallskip

\begin{remark} {\rm Instead of the expression \eqref{deck5} we could have as well said that $T$ acts on $L^\star$-valued $(n-p, n-q)$ forms
smooth in conic sense.}
\end{remark}

\begin{remark} {\rm A current $T$ as in the Definition \ref{currents} is ``more singular" than a $L$-valued current, since it only acts on forms which are smooth in conic sense.}
\end{remark}
\medskip

\noindent The condition \eqref{deck5} is equivalent to the following.
\begin{proposition}\label{currents2}
A $(p, q)$-current $T$ with values in $(L, h_L)$ is given by a collection of
	invariant currents $T_i$ on $U_i$ (where $i\in I$) such that for each compact $K\subset U_i$ there exists a constant $C_K> 0$ and a number $d\in\mathbb N$ such that following holds
	\begin{equation}\label{deck51}
		\left|\int_{U_i}\frac{1}{w_i^{qm}}T_i\wedge \ol \rho_i \right|^2\leq C_K\sum_{s=0}^d \sup_{V_i\setminus Y}|\nabla^s\rho_i|^2   
	\end{equation}
	for all $(n-q, n-p)$-forms $\rho$ which are $L$-valued, with compact support in $V_i$ and smooth in conic sense, where $\displaystyle \rho_i:= \frac{1}{w_i^{mq}}\pi_i^\star(\rho)$.
\end{proposition}	

\begin{proof} This is due to the fact that we have 
\begin{equation}\label{van38}
\sum_{s=0}^d \sup_{V_i\setminus Y}|\nabla^s _{h_{L}}\rho|^2\simeq \sum_{s=0}^d \sup_{U_i}|\nabla^s\rho_i|^2,
\end{equation} 
and one can see by the same arguments as in Proposition \ref{orbderiv}.
\end{proof}
\medskip

\begin{remark} {\rm  We have the following alternative description of $T$ in terms of 
the local currents $(T_i)$.
Let $\phi$ be an $L$-valued form, smooth in conic
sense and let $T$ be a current as in Definition \ref{currents}. Then we have
\begin{equation}\label{deck6}
	\int_XT\wedge \ol\phi e^{-\varphi_L}= \sum_{i\in I}\frac{1}{\delta_i}\int_{U_i}\frac{\theta_i\cdot (T_i\wedge \ol \phi_i)}{w_i^{qm}}
\end{equation}  
where we recall that $\displaystyle \phi_i= \frac{1}{w_i^{mq}}\pi_i^\star(\phi)$ and $(\theta_i)$ is a partition of unit corresponding to $(V_i)_{i\in I}$. The integer
$\delta_i$ is the degree of the map $\pi_i$. Note that $\displaystyle \frac{T_i\wedge \ol \phi_i}{w_i^{qm}}$ is considered as a $(n,n)$-current on $U_i$ by using the smooth metric $h_0$ on $L_0 := L- \sum q_i Y_i$, where $h_0 := h_L e^{\sum q_i\log |s_{Y_i}|^2}$.
}
\end{remark}

\medskip

\subsubsection{Examples}\label{subexample} We are discussing next a few examples.
\smallskip

\noindent $\bullet$  
Let $\Theta$ be a smooth $L$-valued $(p,q)$-form on $X$.  In general, $\Theta$ is not smooth in the conic sense --the case $p=n$ and $q_i< 1$ is exceptional. However, $\Theta$ induces a current in conic sense (via 
its principal value) as follows:
\begin{equation}\label{van411}
	\int_X\Theta\wedge \ol \rho:= \lim_{\ep\rightarrow 0}\sum_i\frac{1}{\delta_i}\int_{U_i}\mu_\ep\frac{\theta_i}{w_i^{qm}}\pi_i^\star\Theta\wedge \ol \rho_i  ,
\end{equation}
where $\theta_i$ is a partition of unity.
\smallskip

\noindent $\bullet$ More generally,  let $E= \sum E_i$ be a divisor on $X$ such that $E+Y$ is snc (so in particular $E$ and $Y$ do not have any common components).  
Let $\Theta$ be an $L$-valued $(p, q)$-form on $X$ with log-poles along $E+Y$ and smooth on $X\setminus (E+Y)$. 
$\Theta$ induces naturally an $(L, h_L)$-valued current as follows: \begin{equation}\label{van4111}
	\int_X\Theta\wedge \ol \rho:= \lim_{\ep\rightarrow 0}\sum_i\frac{1}{\delta_i}\int_{U_i}\mu_\ep \frac{\theta_i}{w_i^{qm}}\pi_i^\star\Theta\wedge \ol \rho_i.
\end{equation}
\smallskip

\noindent We do not detail here the verification of the fact that the equalities \eqref{van411}
and \eqref{van411} define currents in conic setting but rather refer to the Appendix where this is extensively discussed.
\medskip

\smallskip
		
\noindent $\bullet$ Let $E= \sum E_i$ and $W$ be a divisor on $X$ and an additional hypersurface, respectively such that $E+ W$ is snc. If $\alpha$ is an $L$-valued $(n-1, 0)$-form with log-poles along $E|_W$ on $W$, then the corresponding current is defined by the formula   
	\begin{equation}\label{deck8}
		\int_{X}T\wedge \ol \phi:= \sum_i\frac{1}{\delta_i}\int_{W_i}\theta_i\frac{1}{w_i^{qm}}\pi_i^\star\alpha\wedge \ol \phi_i
	\end{equation}
	where $W_i:= \pi_i^{-1}(W\cap V_i)$. Because of the type of $\alpha$ it is immediate to verify that \eqref{deck5} holds true (with $d= 0$).
	
\smallskip
		
\noindent $\bullet$	
 Let $\lambda$ be a smooth, $(n,1)$ form with values in the bundle $L+ kE$. A more general version of the current introduced at first bullet above is given as follows
	\begin{equation}\label{deck9}
		\int_{U_i}\frac{1}{w_i^{qm}}T_i\wedge \ol \phi:= \lim_{\ep\to 0}\frac{1}{d_i}\int_{U_i}\frac{1}{w_i^{qm}}\pi_i^\star\Big(\theta_\ep\frac{\lambda}{s_E^k}\Big|_{V_i}\Big)\wedge \ol \phi
	\end{equation}
	where $\theta_i$ is a sequence of cutoff functions with compact support in $X\setminus E$ and converging to the identity function of this open subset of $X$. By using the argument involving the Taylor expansion of $\theta_\ep$ we infer that the global object $T$ defined by \eqref{deck9} has the following property 
	\begin{equation}\label{deck10}
		\left|\int_XT\wedge \ol \phi e^{-\varphi_L}\right|^2\leq C\sum_{i\in I}\sum_{s=0}^k\sup_{U_i}|\nabla^s\phi_i|^2.
	\end{equation}
	where $\phi$ is an $L$-valued form of type $(n-1, 0)$, smooth in conic sense
	and $\phi_i$ are defined in $\eqref{deck4}$. Again, the verification of \eqref{deck10}
	is detailed at the end of the Appendix.
	\bigskip
	
	\noindent The next statement can be seen as "mise en bouche" for the type of arguments in the Appendix.
	
	\begin{lemme}\label{taylorde} Let $\phi$ be a negative, $\cC^\infty$ real function defined on the unit disk in $\CC$. 
		We define 
		$${\mathcal D}_\ep:= \{z\in \mathbb D : \ep< |z|e^{-\phi(z)}< 2\ep\}, \qquad {\mathcal D}_{\ep, r}:= \{z\in \mathbb D : \ep< |z|e^{-\phi(z)}< r\}.$$
		The following hold true.
		\begin{enumerate}
			\smallskip
			
			\item[\rm (a)] For any smooth function $f$ and any $\delta< 1$ we have $\displaystyle \lim_{\ep\to 0}\int_{{\mathcal D}_\ep}
			f(z)\frac{d\lambda}{z|z|^{2\delta}}= 0$.
			\smallskip
			
			\item[\rm (b)] For any smooth function $f$, any $\delta< 1$ and any fixed $r> 0$ the set of complex numbers $\displaystyle \left(\int_{{\mathcal D}_{\ep, r}}
			f(z)\frac{d\lambda}{z|z|^{2\delta}}\right)_{\ep> 0}$ is bounded.
		\end{enumerate}
	\end{lemme}
	
	\begin{proof} We discuss the point (a). A first remark is that it is enough to show that the limit is question is zero for constant functions $f$ (as we see by writing $f(z)= f(0)+ zf_1(z)+ \ol z f_{\ol 1}(z)$ and observing that for the second and third term things are immediate).
		\smallskip
		
		\noindent We take then $f= 1$; since $\varphi$ is negative, we can write
		\begin{equation}\label{van1}
			\int_{{\mathcal D}_\ep}\frac{d\lambda}{z|z|^{2\delta}}= \int_{\ep< |z|< 2\ep}\frac{d\lambda}{z|z|^{2\delta}}- \int_{2\ep e^{\phi(z)}< |z|< 2\ep}\frac{d\lambda}{z|z|^{2\delta}}+
			\int_{\ep e^{\phi(z)}< |z|< \ep}\frac{d\lambda}{z|z|^{2\delta}}  
		\end{equation}
		and the first term of the RHS is equal to zero. We can assume that $1\leq 2\delta< 2$ and it is enough to show that we have
		\begin{equation}\label{van2}
			\lim_{\ep\to 0}\int_{\ep(1- C\ep)< |z|< \ep}\frac{d\lambda}{|z|^{1+2\delta}}= 0,  
		\end{equation}
		where $C$ is a positive constant depending on the function $\varphi$. This in turn is obvious. 
		\smallskip
		
		\noindent We leave the rest of the proof of the lemma to the interested readers.
		
	\end{proof}
\medskip

\noindent As  application, we obtain the following result. 
Let $\beta$ be an $L$-valued form of type $(n-1,1)$ and logarithmic poles along $E+Y$. Then we have
\begin{equation}\label{add}
	\lim_{\ep\to 0}\int_X\theta_\ep D'\beta\wedge \ol\gamma e^{-\varphi_L}= 0
\end{equation}
for any holomorphic $L^2$ form $\gamma$, where $\theta_\ep$ is the standard cut-off functions.

\smallskip

\section{De Rham-Kodaira decompostion for conic currents}\label{current}

In this subsection, we would like to show that the results of de Rham-Kodaira in \cite{dRKo} concerning Hodge decomposition for currents on compact manifolds have a complete analogue in our setting, i.e. for $(X, \omega_\cC)$ and $(L, h_L)$ in the setting of Subsection \ref{subsectconic}.

\noindent By Theorem \ref{CoHo1}, for any $L$-valued $(p, q)$-form $\phi$, smooth in conic sense, we have 
\begin{equation}\label{deck11}
	\phi= \cH(\phi)+ \Delta''\psi
\end{equation}
where $\cH(\phi)$ is the harmonic part of $\phi$ and $\psi$ is an $L$-valued $(p, q)$-form. If we add to $\psi$ a harmonic form, the equality 
\eqref{deck11} still holds. Thus, as in the standard case we introduce the Green operator.

\begin{defn} We denote by $\cG(\phi)$ 
the unique $L$-valued $(p, q)$-form which is orthogonal to the harmonic space and such that 
\begin{equation}\label{entdeck11}
\phi= \cH(\phi)+ \Delta''\big(\cG(\phi)\big)
\end{equation}
holds true.
\end{defn}

\begin{remark}{\rm Since harmonic forms are in particular smooth 
in conic sense, we infer that $\cG(\varphi)$ has this property as well. 
}\end{remark}
\medskip  

\noindent It is easy to see that $\cG$ commutes with $\dbar$.  As a consequence, if $\dbar \phi= 0$, then $\dbar \cG(\phi)= 0$ and so the equality 
\eqref{deck11} becomes
\begin{equation}\label{deck14}
	\phi= \cH(\phi)+ \dbar S
\end{equation}
where $S:= \dbar^\star \cG(\phi)$. The same is true for forms with value in the dual $L^\star$ of $L$.
\medskip

\noindent Let $T$ be an $(L, h_L)$-valued $(p,q)$-current. Precisely as in the untwisted case \cite{dRKo} (i.e. without any "$L$''), we can define the \emph{harmonic part} of $T$ as
\begin{equation}\label{pave71}
	\cH(T):= \sum_i \langle T, \xi_i\rangle \xi_i
\end{equation}
where $(\xi_i)_i$ is a basis of $L^2$ harmonic $L$-valued forms of type $(p, q)$ and where the notation in \eqref{pave71} is
\begin{equation}\label{pave72}
	\langle T, \xi\rangle:= \int_X T\wedge \ol {\gamma_\xi},
\end{equation}
and $\gamma_\xi:= \star \xi$ and so $\ol {\gamma_\xi} = \sharp \xi$ becomes an $(n-p, n-q)$ form with values in $L^\star$. 
\bigskip

Let $T$ be a $(p, q)$-current with values in $(L, h_L)$, cf. Definition \ref{currents}. We define next by duality the Green operator on $T$ by the formula
\begin{equation}\label{pave17}
\int_X \cG(T) \wedge \ol{\star \phi} := \int_{X}T\wedge\ol{\star \cG (\phi)},
\end{equation}
for any  $L$-valued $(p, q)$-form $\phi$, smooth in conic sense. 
This induces a decomposition of the current $T$ as in the usual case
\begin{equation}\label{pave18}
	T= \cH(T)+ \Delta''\big(\cG(T)\big).
\end{equation} 
\begin{defn} The equality \eqref{pave18} is called
de Rham-Kodaira decomposition of the current $T$. 
\end{defn}
\medskip

\begin{remark}{\rm  If moreover $T$ is closed, then we can write
	\begin{equation}\label{pave19}
		T= \cH(T)+ \dbar T_{1}
	\end{equation}
	where $\displaystyle T_{1}= \dbar^\star\big(\cG(T)\big)$. This can be seen as follows: thanks to \eqref{diffdep}, we know that $\cH(T)$
is $\dbar$-closed. Then we have $\Delta'' \dbar \big(\cG(T)\big)=0$. It follows that $\dbar \big(\cG(T)\big)$ is smooth in orbifold sense and then we get $\dbar^\star \dbar \big(\cG(T)\big) =0$ by the usual integration by parts.}
\end{remark}
\medskip

\noindent We have the following compatibility relation between the notions we have introduced above.

\begin{lemme}\label{commu}
	Let $s$ be a  $L$-valued $(p,q)$-current and $t$ be a $L^\star$-valued $(n-p, n-q)$-form smooth in conic sense. We have 
\begin{equation}\label{com}
\int_X \cG(s)\wedge t = \int_X s\wedge {\cG(t)}.
\end{equation}
\end{lemme}

\begin{proof} 
Note first that, since $\cG (s)$ is orthogonal to the harmonic space, by Proposition \eqref{sharppair}, we have 
$$\int_X \cG (s)\wedge \cH (t) =0 .$$
Then we have 
\begin{equation}\label{com1}
\int_X \cG (s)\wedge t = \int_X \cG(s) \wedge \Delta^{''} (\cG{t})= 
 \int_X \Delta^{''} (\cG(s)) \wedge \cG (t) =  \int_X s\wedge {\cG(t)}.
\end{equation}
\end{proof}

\subsection{Regularity of the Green kernel for certain currents}

\noindent Coming back to the case of a current induced by a form $\alpha$ with log-poles,
the Green current $\cG(\alpha)$ has a few interesting properties that we collect in our next statement.

\begin{proposition}\label{reg2}
	Let $E=\sum E_i$ be divisor on $X$ such that $E+Y$ is snc.
	Let $\alpha$ be a smooth $L$-valued $(n, q)$-form on $X\setminus (E\cup Y)$, with log pole along $E+Y$. The following statements hold true.
	\begin{enumerate}
		\smallskip
		
		\item[\rm (1)] The current $\cG(\alpha)$ is a smooth $L$-valued $(n, q)$--form on
		$X\setminus (E\cup Y)$ which moreover is $L^2$ on $X$.  
		\smallskip
		
		\item[\rm (2)] For each index $i$ the local form 
		$\displaystyle  v_i:= \frac{1}{\prod w_i^{q_im}}\pi_i^\star\left(\cG(\alpha)\right)$   
		is a solution of an equation of the following type
		\begin{equation}\label{pave8}\nonumber
			\Delta''_{sm}(v_i)= u_i  
		\end{equation}
		where $u_i$ is a $(n,q)$-form with logarithmic poles along $E+Y$.
		\smallskip
		
		\item[\rm (3)] For every $\delta> 0$ form $v_i$ belongs to the Sobolev space $W^{2, 2-\delta}$ (i.e. it has two derivatives in $L^{2-\delta}$).
		\smallskip
		
		\item[\rm (4)] The form $\tau_i:= \partial^\star v_i$ is $L^2$ and moreover
		``smooth'' with respect to the variables $w_j$ for each $j=r+1,\dots, n$, in the sense that it admits a Taylor expansion
		\begin{equation}\label{pave7}\nonumber
			\tau_i(w)= \sum_{|I|+|J|\leq N-1} w^{\prime \alpha}\ol w^{\prime \beta}\tau_{I\ol J}(w'')+
			\sum_{|I|+|J|= N} w^{\prime \alpha}\ol w^{\prime \beta}\tau_{I\ol J}(w)  
		\end{equation}
		where $$w'= (w_{r+1},\dots, w_{n}),\qquad w''= (w_{1},\dots, w_r)$$
		and $\tau_{I\ol J}$ are $L^2$.
	\end{enumerate} 
\end{proposition}

\begin{proof} 
In order to prove the point (1), our first remark is that 
	
	\begin{equation}\label{pave10}
		\int_X |\cG(\Phi)|^2e^{-\varphi}\leq C \int_X|\Phi- \cH (\Phi)|^2e^{-\varphi}
	\end{equation}
	by Poincar\'e inequality. The inequality \eqref{pave10} is valid for $(n, q)$-forms and it implies that 
	\begin{equation}\label{pave11}
		\int_X |\cG(\Phi)|^2e^{-\varphi}\leq C \int_X|\Phi|^2e^{-\varphi}
	\end{equation}
	because the $L^2$ norm of $\cH (\Phi)$ is smaller than the $L^2$ norm of $\Phi$. The smoothness of 
	$\cG(\alpha)$ in the complement of the support of the divisor $E+Y$ is standard. Then the Sobolev norm of order two of $\cG(\Phi)$ is bounded by the RHS of \eqref{pave11}.
	\smallskip
	
	\noindent By the definition of the current $\cG(\alpha)$ as in \eqref{pave17} it follows that we have
	\begin{equation}\label{pave60}
		\left|\langle \cG(\alpha), \ol \Phi\rangle\right|\leq \Vert\alpha\Vert_{L^{2-\delta}}\Vert \cG(\Phi)\Vert_{L^{\frac{2-\delta}{1-\delta}}}.
	\end{equation}
	Now the $L^{\frac{2-\delta}{1-\delta}}$-norm of $\cG(\Phi)$ is certainly controlled by the Sobolev norm above (at least when $\delta$ is small enough) so in conclusion we get
	\begin{equation}\label{pave61}
		\left|\langle \cG(\alpha), \ol \Phi\rangle\right|^2\leq C_\delta \Vert\Phi\Vert^2_{L^2}
	\end{equation}
	which shows that $\cG(\alpha)$ is in $L^2$.
	\smallskip
	
\noindent The first part of the point (2) (i.e. the equation satisfied by the form $v_i$) was already discussed in Proposition \ref{orbderiv}. The fact that the RHS is a form with log-poles is immediate and we skip the details.
\smallskip
	
	\noindent The points (3) and (4) are a bit more interesting and we will provide a complete treatment.
	
	\noindent To start with, the arguments provided so far allow us to assume that the form $v_i$ in \eqref{pave8} is in $L^2$.
	Together with the G\aa rding inequality established in \cite{Agm}, Theorem 7.1 we obtain  
	\begin{equation}\label{pave62}
		\Vert v_i\Vert_{W^{2, 2-\delta}}\leq C_p(\Vert v_i\Vert_{L^{2-\delta}}+ \Vert u_i\Vert_{L^{2-\delta}})
	\end{equation}
	for any $0<\delta<2$.  The point $(3)$ is thus proved.
	
	\bigskip
	
	\noindent To prove the point $(4)$,  let $\Xi$ be any local vector field with constant coefficients of the following shape
	\begin{equation}\label{pave63}
		\Xi= \sum_{i=r+1}^n a_i\frac{\partial}{\partial w_i}.
	\end{equation}
	The equation \eqref{pave8} implies that we have 
	\begin{equation}\label{pave64}
		\Delta''_{sm}(\Xi\cdot v_i)= \Xi\cdot u_i + \mathcal L(v_i)
	\end{equation}
	where $\mathcal L$ is a second order operator. In \eqref{pave64} we denote by $\Xi\cdot v_i$ the form whose coefficients are the derivative of those of $v_i$ in the direction of $\Xi$. 
	\smallskip
	
	\noindent 
	By Sobolev embedding theorem we infer that $\Xi(v)$ is in $L^2$ (provided that $\delta\ll 1$).
	It follows from \eqref{pave64} that $\Xi(v)$ has two derivatives in $L^p$. We can iterate this method, 
	and infer that $\Xi^m(v)\in L^2$ for any $m$.
	\medskip
	
	\noindent Finally, the usual argument --by using a Fourier expansion-- implies that $v_i$ admits a Taylor expansion as required in (4). We discuss here only a particular case, for notational reasons.
	Let $f\in L^2(\mathbb D)$ be a square integrable function defined on the unit disk in $\R^2$, and such that
	$\displaystyle \frac{\partial^m}{\partial x^m}(f)$ is in $L^2$ for any positive integer $m$.
	
	We consider the convolution $f_\ep$ of $f$ with the usual regularisation kernel and as usual we can assume that this functions are 1-periodic. We then write
	\begin{equation}\label{pave75}
		f_\ep(x, y)= \sum_{k\in \Z}f_{\ep, k}(y)e^{\sqrt{-1}kx}
	\end{equation}
	and by hypothesis, together with the $L^2$ convergence of convolutions and their derivatives
	we infer in particular that for any $m\in \Z_+$ we have
	\begin{equation}\label{pave76}
		\sum_{k\in \Z}k^{2m}\int_{|y|< 1}|f_{\ep, k}(y)|^2d\lambda(y)< C_m
	\end{equation}
	where $C_m$ is a constant independent of $\ep$.
	\smallskip
	
	\noindent A first consequence of \eqref{pave76} is the existence of $C>0$ such that 
	\begin{equation}\label{pave77}
		\int_{|y|< 1}|f_{\ep}(0, y)|^2d\lambda(y)< C
	\end{equation}
	which can be seen by observing that $$\displaystyle \left|\sum_{k=-N}^Nf_{\ep, k}(y)\right|^2\leq
	2\sum_{k=-N}^Nk^2|f_{\ep, k}(y)|^2\sum_{k= 1}^N\frac{1}{k^2}$$
	and integrating. 
	\smallskip
	
	\noindent We next write $\displaystyle f_\ep(x, y)= f_\ep(0,y)+ x\int_0^1\partial_1f_{\ep}(tx, y)dt$. Remark that we already know that the $L^2$ norm of the first term of the RHS is uniformly bounded,
	hence it is weakly convergent towards a function in $L^2$. We will analyse next the other function; notice that 
	\begin{equation}\label{pave78}
		\int_\mathbb D\left|\int_0^1\partial_1f_{\ep}(tx, y)dt\right|^2d\lambda\leq \int_0^1dt\int_\mathbb D\left|\partial_1f_{\ep}(tx, y)\right|^2d\lambda
	\end{equation}
	which equals 
	\begin{equation}\label{pave79}
		\int_\mathbb D\log\frac{1}{|x|}\left|\partial_1f_{\ep}(x, y)\right|^2d\lambda.
	\end{equation}
	By H\"older inequality the expression \eqref{pave79} is smaller than
	\begin{equation}\label{pave800}
		C_\delta\int_{|y|< 1}dy\left(\int_{|x|< 1}\left|\partial_1f_{\ep}(x, y)\right|^{2+ 2\delta}dx\right)^{\frac{1}{1+\delta}},
	\end{equation}
	and on the other hand, for each fixed $y$ we have 
	\begin{equation}\label{pave811}
		\frac{1}{C}\left(\int_{|x|< 1}\left|\partial_1f_{\ep}(x, y)\right|^{2+ 2\delta}dx\right)^{\frac{1}{1+\delta}}\leq 
		\int_{|x|< 1}\left|\partial_1f_{\ep}(x, y)\right|^{2}dx+ \int_{|x|< 1}\left|\partial_{11}f_{\ep}(x, y)\right|^{2}dx
	\end{equation}
	by Sobolev inequality. The integration with respect to $y$ of \eqref{pave811} shows that the LHS of 
	\eqref{pave78} is uniformly bounded. 
	
	\noindent The arguments for the general case are completely similar --modulo the fact that the number of indexes is greater-- and we will not provide any further detail. 
\end{proof}
\bigskip

\noindent Let $E=\sum E_i$ be divisor on $X$ such that $E+Y$ is snc. Consider next a current of the following type 
\begin{equation}\label{int111}
	T:= \frac{\lambda}{s_E}- \sum_i T_{\gamma_i} ,
\end{equation}
where $\lambda$ is a $L+E$-valued form, smooth in conic sense, $\gamma_i$ is a $L$-valued form on $E_i$, smooth in conic sense, and $T_{\gamma_i}$ is the current on $X$ defined by $\gamma$, cf. Subsection \ref{subexample}. 
We establish next some of the properties of $\cG(T)$. We will namely discuss the following result.
\begin{proposition}\label{reg3}
	Let $T$ be the current defined in \eqref{int111}. 
	The associated current $\cG(T)$ is smooth in the complement of the divisor $E+F$ and it is moreover $L^2$, i.e.
	\begin{equation}\label{int21}
		\int_X|\cG(T)|_{\omega_\cC}^2e^{-\varphi}dV_{\omega_\cC}< \infty.
	\end{equation}
\end{proposition}

\begin{proof}
	We first remark that $T$ is more singular than e.g. a current induced by a form with log poles, so it is natural that the regularity properties of $\cG(T)$ are considerably weaker than 
	those established in Proposition \ref{reg2}.  
	
	\noindent By definition we have
	\begin{equation}\label{int3}
		\langle \cG(T), \Phi\rangle= \int_X\frac{\lambda}{s_E}\wedge \ol{\cG(\Phi)} e^{-\varphi_L}+ 
		\sum_i\int_{E_i}\gamma_i\wedge \ol{\cG(\Phi)} e^{-\varphi_L},
	\end{equation}
	and the absolute value of the first part of the sum is bounded by $\displaystyle \Vert\Phi\Vert_{L^2}$ up to a constant: this is a consequence of the proof of the first point of Proposition \ref{reg2}.
	
	\medskip
	
	\noindent We  treat now the second part of the RHS of \eqref{int3}. Let $\Gamma_i$ be a $L^2$ and \emph{smooth} extension of the $L$-valued form $\gamma_i$. 
	This is constructed e.g. by a partition of unit --we notice the $L^2$ condition is fulfilled precisely because $E+Y$ is a snc divisor.  
	For each $i$ we can write
	\begin{equation}\label{int41}
		\int_{E_i}\gamma_i\wedge \ol{\cG(\Phi)} e^{-\varphi_L}= \int_X\left(\theta_i+ dd^c\log|s_i|^2\right)\wedge\Gamma_i\wedge \ol{\cG(\Phi)} e^{-\varphi_L}.
	\end{equation}
	where $\theta_i$ is the curvature of the bundle associated to the divisor $E_i$ endowed with a non-singular metric. 
	
	\noindent The first term of the RHS of \eqref{int41} is bounded by the $L^2$ norm of $\Phi$, given that 
	$\theta_i\wedge \Gamma_i$ is $L^2$ with respect to $h_L$. 
	
	\noindent For the second term of the RHS of \eqref{int41}, we are using Stokes and we have to estimate
	\begin{equation}\label{int5}
		\int_X\frac{\partial s_i}{s_i}\wedge\dbar\Gamma_i\wedge \ol{\cG(\Phi)}e^{-\varphi_L}, \qquad
		\int_X\frac{\partial s_i}{s_i}\wedge\Gamma_i\wedge \ol{D'\cG(\Phi)}e^{-\varphi_L}.
	\end{equation} 
	To this end we simply use the G\aa rding inequality (after composing with the local ramification maps)
	and infer that 
	\begin{equation}\label{int6}
		\Vert \cG(\Phi)\Vert^2_{W^{2,2}}\leq C\left(\Vert\Phi\Vert^2_{L^2}+ \Vert\cG(\Phi)\Vert^2_{L^2}\right).
	\end{equation}
	By Sobolev inequality the $L^{2+\delta}$-norm of $\cG(\Phi)$ and $D'\cG(\Phi)$ are bounded by 
	the left hand side of \eqref{int6} provided that $\delta\ll1$. This ends the proof of our proposition. 
\end{proof}


$$$$

\section{Proof of Theorem \ref{ddbar} and Theorem \ref{fujino}}\label{007'}

\subsection{A general vanishing theorem}\label{optimum}

\noindent The plan for this section is to prove Theorem  \ref{ddbar} and Theorem \ref{fujino}.  
We prove the Theorem \ref{ddbar} in this subsection. As an application, we will prove Theorem \ref{fujino} in the next subsection.  We would like to prove:

\begin{thm}\label{ddbar01}{\rm [= Theorem \ref{ddbar}]}
 Let $X$ be a $n$-dimensional compact Kähler manifold and let $E$ be a snc divisor on $X$, $s_E$ be the canonical section of $E$.   Let $(L, h_L)$ be a holomorphic line bundle on $X$ such that
 $$i\Theta_{h_L} (L) = \sum q_i [Y_i] + \theta_L,$$ where $q_i \in ]0,1[ \cap \mathbb Q$, $E +\sum Y_i$ is snc, and the form $\theta_L$ is smooth, semi-positive.
 
 Let $\lambda$ be a $\dbar$-closed smooth $(n,q)$-form with value in $L+E$. If there exists $\beta_1$ and $\beta_2$, two $L$-valued $(n-1, q)$-form and $(n-1, q-1)$-form with log poles along $E+\sum Y_i$ respectively, such that
 \begin{equation}
 	\frac{\lambda}{s_E} = D' _{h_L}\beta_1 + \theta_L \wedge \beta_2 \qquad\text{ on } X\setminus (E+\sum Y_i) ,
 \end{equation}	
 then the form $\lambda$ is $\dbar$-exact, i.e, 
 the class $[\lambda]=0 \in H^q (X, K_X +L+E)$.
\end{thm}

Remark that
the case $E=0$ is quite simple: it follows from Theorem \ref{Hodge} and standard argument.
\medskip



\begin{proof}[Proof of Theorem \ref{ddbar01}]
  We consider a test form $\displaystyle \varphi \in C^\infty _{(0,n-q)} \big(X,  -L - E\big)$ with compact support in $X\setminus Y$.

  \noindent Then we have the orthogonal decomposition
\begin{equation}\label{van3}
  \varphi =\varphi_1 +\varphi_2, \qquad \varphi_1\in \ker \dbar,\quad
  \varphi_2\in (\ker \dbar)^{\perp} 
\end{equation}
where the reference metric on $X$ is $\omega_{\mathcal C}$. The bundle $-L$ is endowed with the dual metric $h_L ^*$ and $-E$ with a smooth, arbitrary metric.
\smallskip

\noindent Our proof relies on the following two statements: 
\begin{equation}\label{check1}
	\int_X s_E(D' _{h_L}\beta_1 + \theta_L\wedge \beta_2)\wedge \varphi_1 =0
\end{equation}	
and moreover there exists a constant $C> 0$ such that 
\begin{equation}\label{check2}
 \int_X |\varphi_2|^2 dV_{\cC} \leq C \int_X |\dbar\varphi_2|^2 dV_{\cC}= C\int_X |\dbar\varphi|^2dV_{\cC}.
\end{equation}	
To prove \eqref{check1}, note first that, thanks to Corollary \ref{Kompsup}, the form $\varphi_1$ is smooth in conic sense. Then $s_E\cdot \varphi_1$ has the same property. 
Moreover, $\dbar (s_E\varphi_1 )=0$ and $s_E\varphi_1 $ vanishes along $E$. 
Then we obtain \eqref{check1} as consequence of Lemma \ref{keylemm} below. For \eqref{check2}, we simply use Theorem \ref{CoHo1}
which implies that the image of $\dbar$ operator is closed (regardless to the fact that the curvature of $L^\star$ is negative). 
\medskip

\noindent We consider next the linear application
$$\dbar\varphi \rightarrow \int_X s_E(D' _{h_L}\beta_1 + \theta_L \wedge \beta_2)\wedge \varphi.$$
It follows from  \eqref{check1} and \eqref{check2} that it
is bounded. By standard functional analysis there exists a form $f\in L^2 _{(n,q-1)}(X, L +E )$ such that 
$$\int_X s_E(D' _{h_L}\beta_1 + \theta_L \wedge \beta_2)\wedge \varphi  = \int_X \dbar\varphi  \wedge f .$$
Therefore we infer that the equality
\begin{equation}\label{van5}
 \lambda= s_E( D' _{h_L}\beta_1 + \theta_L\wedge \beta_2)= \dbar f \qquad\text{ on } X\setminus Y.
\end{equation}
Since the LHS in
\eqref{van5} is smooth, it is $L^2$. Moreover, $f$ is also $L^2$. By using the remark after Theorem \ref{JP}, we know that \eqref{van5} holds on the total space $X$. 
Theorem \ref{ddbar01} is thus proved modulo the lemma \ref{keylemm} below. 
\end{proof}
\medskip

\noindent The following key lemma \ref{keylemm} is a generalization of \cite[Lemma .2.3]{LRW} in our setting.
Prior to stating it, we introduce a further notation. Let $\gamma$ be an $L$-valued form of $(p, q)$-type which moreover has log-poles along $E$ and it is $\dbar$-closed when restricted to the complement of the support of $E$. Locally on a coordinate subset $\Omega\subset X$ such that $E_1\cap \Omega= (z_1= 0)$ we can write
\begin{equation}\label{van6}
\gamma|_{\Omega}= \frac{dz_1}{z_1}\wedge \gamma_1+ \gamma_2
\end{equation}
where $\gamma_1$ and $\gamma_2$ are smooth along $E_1\cap \Omega$. We define
\begin{equation}\label{van7}
\Res_{E_1}(\gamma)|_{\Omega}:= \gamma_1|_{\Omega}
\end{equation}
and it is not difficult to see that the following hold:
\smallskip

\noindent $\bullet$ The forms $\Res_{E_1}(\gamma)|_{\Omega}$, even if given locally as in
\eqref{van7}, defines a global $(p-1, q)$-form on $E_1$ with values in $\displaystyle L|_{E_1}$
and log-poles on $\displaystyle (E-E_1)|_{E_1}$. We will denote it with $\Res_{E_1}(\gamma)$
\smallskip

\noindent $\bullet$ We have
 \begin{equation}\label{res}
 	\Res_{E_1}(\gamma)= -\chi_{E_1}\dbar \gamma
 	\end{equation} i.e. residue along $E_1$ can be obtained by multiplying the current $\dbar \gamma$ with the characteristic function of $E_1$.

\smallskip

\noindent $\bullet$ The equality $\dbar \Res_{E_1}(\gamma)= 0$ holds in the
complement of the support of $\displaystyle (E-E_1)|_{E_1}$.
\medskip

\noindent All the properties above are completely standard. In the specific context
of Theorem \ref{ddbar}, we consider the current 
 \begin{equation}\label{vanres1}
 \wh\Theta:= D' _{h_L}\beta_1 + \theta_L\wedge \beta_2   
 \end{equation}
and its residues 
along the components of $E$. As a preparation for the proof of this result, we discuss next a few properties of $\wh\Theta$ that we will next use.
\smallskip

\noindent $\bullet$ The first fact to remark is that we have the equality
\begin{equation}\label{vanres44}
\frac{\lambda}{s_E}= D'_{h_L}\beta_1 + \theta_L\wedge \beta_2   
\end{equation}
in the sense of currents on the total space $X$. This is due to the fact that the current
$D' _{h_L}\beta_1$ has no divisorial part supported in $E+Y$. Note that this is absolutely not the case in general if we replace $D'_{h_L}$ with $\dbar$.
\smallskip

\noindent $\bullet$ Another important observation is that the residue of $\wh\Theta$ along any component of $E$ has the same ``shape''. i.e. it belongs to the
image of $D'_{h_L} + \theta_L \wedge $. This is a simple calculation that we skip.

\smallskip

\noindent $\bullet$ The residue of $\wh\Theta$ on the components of $Y$ is equal to zero.
This claim may look strange/suspicious, since the
forms $\beta_i$ could have log-poles on $Y$,
so the residues of each form $D' _{h_L}\beta_1$ and $\theta_L\wedge \beta_2$
could be non-zero on some of the components of $Y$. However, the equality
\eqref{vanres44} clarifies the matter.
\medskip

\noindent We have the next auxiliary result which will be very useful in the
proof of Lemma \ref{keylemm} below.
\begin{lemme}\label{dddlem}
  Let $\gamma$ be a $\Delta^{''}_{\omega_{\mathcal C}}$-harmonic $(0, n-q)$-form with values in the dual bundle $L^\star$ of $L$. Then the equality
\begin{equation}\label{vanres4}
\int_X(D' _{h_L}\beta_1 + \theta_L\wedge \beta_2)\wedge \gamma= 0
\end{equation}
holds (here as everywhere, the integral is computed in the complement of $E+Y$as a matter of fact, the form in \eqref{vanres4} is absolutely integrable). 
\end{lemme}
\begin{proof}
Since $\gamma$ is harmonic, it is in particular smooth in conic sense and moreover, as explained above, the equality $\displaystyle D'_{h_L}\beta_1 + \theta_L\wedge \beta_2=\frac{\lambda}{s_E}$ holds in the sense of currents. It follows that the form
under the integral \eqref{vanres4} is absolutely integrable: this is seen as follows. Let 
$\pi:U\to V$ be a ramified cover adapted to our setting and coordinates on $V$ such that 
$$Y\cap V=(z_1\dots z_r=0), \qquad E\cap V=(z_{r+1}\dots z_{r+s}=0).$$
We have 
$$\frac{1}{w^qm}\pi^\star\Big(\frac{\lambda}{s_E}\Big|_V\Big)\wedge \tau= \frac{\O(1)}{\prod w_j^{q_jm+1-m}\prod w_i}d\lambda(w)$$
for any smooth form $\tau$ of type $(0, n-1)$, so our claim follows.
\smallskip

\noindent Thus, it would be enough to show that we have
\begin{equation}\label{vanres5}
\lim_{\ep\to 0}\int_X\mu_\ep(D' _{h_L}\beta_1 + \theta_L\wedge \beta_2)\wedge \gamma= 0
\end{equation}
where $(\mu_\ep)$ is the usual family of truncation functions. Now $\gamma$ is harmonic, so we have
\begin{equation}
D'\gamma= 0, \qquad \theta_L\wedge \gamma=0.
\end{equation}
Therefore, it all boils down to the next equality 
\begin{equation}\label{vanres6} 
\lim_{\ep\to 0}\int_X\partial \mu_\ep\wedge\beta_1\wedge \gamma= 0
\end{equation}
which in turn is clear by a local calculation (using the fact that $\gamma$ is smooth in conic sense).
\medskip

For example, say that $n=2$, so that $\beta_1$ is a $(1,1)$-form and $Y= (z_1=0), E= (z_2=0)$. Then the expression \eqref{vanres6} is locally equal to
\begin{equation}\label{vanres46} 
\int_V\frac{\rho_\ep'(z)}{\log(|z_1|^2e^{-\phi})}\frac{dz_1}{z_1}\wedge \frac{dz_2}{z_2}\wedge(ad\ol z_1+ bd\ol z_2)\wedge\gamma|_V
\end{equation}
and by the local transformation $(w_1, w_2)\to (w_1^m, w_2)$ this expression above becomes
\begin{equation}\label{vanres47} 
\int_U\frac{\rho_\ep'(w)}{\log(|w_1|^2e^{-\phi})}\frac{dw_1}{w_1}\wedge \frac{dw_2}{w_2}\wedge(a_1d\ol w_1+ b_1d\ol z_2)\wedge\frac{\tau}{w_1^{mq}}
\end{equation}
where $a_1, b_1$ are smooth functions with compact support on $U$ and $\tau$ is a smooth $(0,1)$ form. This is a 
consequence of the fact that harmonic forms are smooth in conic sense. Arguments as in Lemma \ref{taylorde} and the examples above do apply, so we get
\begin{equation}\label{vanres48} 
\lim_{\ep\to 0}\int_U\rho_\ep'(w)\frac{f(w)}{w_1^{1+mq}w_2}\frac{d\lambda}{\log(|w_1|^2e^{-\phi})}= 0.
\end{equation}
Actually we refer to the Appendix for a complete proof of \eqref{vanres48}. 
\end{proof}  

\medskip

\noindent Before stating our main lemma, we introduce some notations.
We suppose that the snc divisor $E$ is written as $E= \sum_{1\leq i\leq r} E_i$.
Let $I_d \subset \{1, \cdots , r\}$ with $|I_d|=d$. 
Let $\alpha$ be a $L^*$-valued smooth form (in the conic sense with respect to $h_L^\star, \omega_{\mathcal{C}}$) defined on some space which contains $E_{I_d}$. 
We define 
$$L_{I_d} (\alpha) :=  \dbar^{\star_{I_{d}}} \cG_{I_{d}}  (\alpha |_{E_{I_d}} ) ,$$ 
where $\cG_{I_{d}} $ is the Green operator on $E_{I_d}$ and $\dbar^{\star_{I_{d}}} $ is the $\dbar^\star$-operator on $E_{I_d}$ with respect to $h_L^\star, \omega_{\mathcal{C}}$.
If $d=0$, $L_{I_{0}} \varphi  := \dbar^\star \cG (\varphi)$ on $X$. 
For every sequence $I_k \subset I_{k-1} \subset \cdots \subset I_0$ fixed, the residue $\Res_{E_{I_k}} (\wh \Theta) $ is defined by the following way: 
\begin{equation}\label{reshigher}
	\Res_{E_{I_k}} (\wh \Theta)  := \Res_{E_{I_k}} \Big(\Res_{E_{I_{k-1}}} \big( \cdots \Res_{E_{I_1}} (\wh \Theta) \big)\Big) .
\end{equation}
Now we can state the key lemma:

\begin{lemme}\label{keylemm}
In the setting of Theorem \ref{ddbar01}, set $\wh\Theta := D' _{h_L}\beta_1 + \theta_L\wedge \beta_2$.   Let $\varphi$ be a $\dbar$-closed $(0,n-q)$-form on $X$ with values in the dual of $L$, which moreover is smooth in conic sense. We assume that $\varphi$ vanishes along $E$. 
Then for every $k\in \mathbb N^*$,  we have 
\begin{equation}\label{maindd}
	 \sum_{I_k \subset I_{k-1} \subset \cdots \subset I_0} \int_{E_{I_k}} \Res_{E_{I_k}} (\wh \Theta)  \wedge  L_{I_{k-1}} \circ L_{I_{k-2}} \cdots \circ L_{I_{0}} \varphi =0 ,
 \end{equation}
and
\begin{equation}\label{addnew}
	\int_{X} \wh \Theta  \wedge  \varphi =0 .
	\end{equation}
\end{lemme}

\begin{proof}
We prove it by induction on $k$ from large to small. For $k=n+1$, it is automatically true since $E_{I_{n+1}}$ is empty. We suppose that the lemma holds for $k+1$. Now we prove the lemma for $k$. 

First of all, we have
\begin{equation}\label{van311}
\int_{E_{I_k}}\Res_{E_{I_k}} (\wh \Theta) \wedge \gamma =0
\end{equation}
for any harmonic $\gamma \in C^{\infty} _{(0, n-k-1)} (E_{I_k}, L^\star)$. Indeed, we know that the residue $\Res_{E_{I_k}} (\wh \Theta)$ obtained by iterating the construction we have recalled before the lemma \ref{dddlem} can be written as
$$\Res_{E_{I_k}} (\wh \Theta)= D'(\beta_k)+ \theta_L\wedge \alpha_k$$
for some forms $\beta_k, \alpha_k$ with log poles. Then Lemma \ref{dddlem} implies \eqref{van311}.
\medskip

\noindent Then the deRham-Kodaira decomposition implies that the equality
$$\Res_{E_{I_k}} (\wh \Theta)  = \dbar \dbar^{\star_{I_k}} \cG_{I_k} (\Res_{E_{I_k}} (\wh \Theta))+ \dbar^{\star_{I_k}}  \dbar \cG_{I_k} (\Res_{E_{I_k}} (\wh \Theta))$$
holds. As a consequence, we obtain
$$\sum_{I_k \subset I_{k-1} \subset \cdots \subset I_0} \int_{E_{I_k}} \Res_{E_{I_k}} (\wh \Theta)  \wedge  L_{I_{k-1}} \circ L_{I_{k-2}} \cdots \circ L_{I_{0}} \varphi = A +B ,$$
where 
$$A : =\sum_{I_k \subset I_{k-1} \subset \cdots \subset I_0}  \int_{E_{I_k}} \dbar^{\star_{I_k}} \dbar  \cG_{I_k} (\Res_{E_{I_k}} (\wh \Theta)) \wedge L_{I_{k-1}} \circ \cdots \circ L_{I_{0}} \varphi$$
and
$$B:=\sum_{I_k \subset I_{k-1} \subset \cdots \subset I_0}   \int_{E_{I_k}} \dbar \dbar^{\star_{I_k}} \cG_{I_k} (\Res_{E_{I_k}} (\wh \Theta)) \wedge L_{I_{k-1}} \circ  \cdots \circ L_{I_{0}} \varphi .$$
We would like to prove that $A=B=0$.
\medskip

\noindent We first estimate $A$. Since $\cG$ commutes with $\dbar$, together with the relation \eqref{res}, we have
$$A= \sum_{I_k \subset I_{k-1} \subset \cdots \subset I_0}  \int_{E_{I_k}} \dbar^{\star_{I_k}} \cG_{I_k} \dbar (\Res_{E_{I_k} } (\wh \Theta)) \wedge L_{I_{k-1}} \circ  \cdots \circ L_{I_{0}} \varphi$$
$$= - \sum_{I_{k+1} \subset I_{k} \subset \cdots \subset I_0}  \int_{E_{I_{k+1}}}  \Res_{E_{I_{k+1}}} (\wh \Theta) \wedge L_{I_k} \circ L_{I_{k-1}} \circ  \cdots \circ L_{I_{0}} \varphi =0.$$
where the last equality is a consequence of the induction hypothesis.

\medskip

\noindent Next, we estimate the quantity $B$. We have
$$B= \sum_{I_k \subset I_{k-1} \subset \cdots \subset I_0} \int_{E_{I_k}} \dbar \dbar^{\star_{I_k}} \cG_{I_k} (\Res_{E_{I_k}} (\wh \Theta)) \wedge L_{I_{k-1}} \circ \cdots \circ L_{I_{0}} \varphi  $$
$$= \sum_{I_k \subset I_{k-1} \subset \cdots \subset I_0} \int_{E_{I_k}} \dbar^{\star_{I_k}} \cG_{I_k} (\Res_{E_{I_k}} (\wh \Theta)) \wedge \dbar (L_{I_{k-1}} \circ \cdots \circ L_{I_{0}} \varphi)   .$$
By applying the deRham-Kodaira decomposition \eqref{pave18} to $(L_{I_{k-2}} \circ \cdots \circ L_{I_{0}} \varphi ) |_{E_{I_{k-1}}}$, we obtain
$$ \dbar (L_{I_{k-1}} \circ \cdots \circ L_{I_{0}} \varphi)  = \gamma +  (L_{I_{k-2}} \circ \cdots \circ L_{I_{0}} \varphi) |_{E_{I_{k-1}}} +  L_{I_{k-1}} \circ \dbar (L_{I_{k-2}} \cdots \circ L_{I_{0}} \varphi)   , $$
where $\gamma$ is a $\Delta''_{E_{I_{k-1}}}$-harmonic $(0,n-k+1)$-form on $E_{I_{k-1}}$ with values in the restriction of the dual of $L$.
\medskip

\noindent We next estimate the integral involving the first two terms.
\smallskip

\noindent $\bullet$ {\sl The term involving} $\gamma$.  Since $\gamma$ is a harmonic $(0,n-k+1)$-form on $E_{I_{k-1}}$
with values in $L^\star$, it follows that $D'\gamma=0$ and $\displaystyle \theta_L|_{E_{I_{k}}}\wedge \gamma=0$ on $E_{I_{k-1}}$. These 
crucial equalities obviously are still verified by restriction one level deeper, to $E_{I_{k}}$. We infer that the restriction
$$\gamma|_{E_{I_{k}}}$$
is still harmonic and 
thus 
$$\int_{E_{I_k}} \dbar^{\star_{I_k}} \cG_{I_k} (\Res_{E_{I_k}} (\wh \Theta)) \wedge \gamma =0.$$
\smallskip

\noindent $\bullet$ {\sl The term involving} $L_{I_{k-2}} \circ \cdots \circ L_{I_{0}} \varphi $. The sum 
$$ \sum_{I_k \subset I_{k-1} \subset \cdots \subset I_0} \int_{E_{I_k}} \dbar^{\star_{I_k}} \cG_{I_k} (\Res_{E_{I_k}} (\wh \Theta)) \wedge  (L_{I_{k-2}} \circ \cdots \circ L_{I_{0}} \varphi)  $$
equals zero as soon as $k-1\geq 1$:  for every $I_{k}\subset I_{k-2}$ fixed, $I_{k-2} \setminus I_k$ contains only two elements. By construction of $\Res_{I_k}$ \eqref{reshigher}, the alternate sum for $I_{k-1}$ satisfying $I_k \subset I_{k-1} \subset I_{k-2}$ equals to $0$.
\medskip

\noindent It follows that if $k-1 \geq 1$, we obtain
$$B=
\sum_{I_k \subset I_{k-1} \subset \cdots \subset I_0} \int_{E_{I_k}} \dbar^{\star_{I_k}} \cG_{I_k} (\Res_{E_{I_k}} (\wh \Theta)) \wedge L_{I_{k-1}} \circ \dbar (L_{I_{k-2}} \circ \cdots \circ L_{I_{0}} \varphi).$$
\medskip 

\noindent Now we can repeat the same argument for $\dbar (L_{I_{k-2}} \circ \cdots \circ L_{I_{0}} \varphi)$ as long as $k\geq 3$.
The only difference is to treat the integral involving the harmonic forms $\gamma$ on $E_{I_{k-2}}$: we use the fact that $\cG_{I_{k-1}} (\gamma|_{E_{I_{k-2}}}) =0$. We finally get 
$$
B = \sum_{I_k \subset I_{k-1} \subset \cdots \subset I_0} \int_{E_{I_k}} \dbar^{\star_{I_k}} \cG_{I_k} (\Res_{E_{I_k}} (\wh \Theta)) \wedge L_{I_{k-1}} \circ \dbar (L_{I_{k-2}} \circ \cdots \circ L_{I_{0}} \varphi)$$
$$\cdots = 
\sum_{I_k \subset I_{k-1} \subset \cdots \subset I_0} \int_{E_{I_k}} \dbar^{\star_{I_k}} \cG_{I_k} (\Res_{E_{I_k}} (\wh \Theta)) \wedge L_{I_{k-1}} \circ L_{I_{k-2}} \circ \cdots \circ L_{I_{1}}\circ \dbar (L_{I_{0}}  (\varphi)) .$$
\smallskip

\noindent Concerning the  term $\dbar (L_{I_{0}}  (\varphi)) = \dbar \dbar^\star \cG (\varphi)$, we consider the Hodge decomposition
$$\dbar \dbar^\star \cG (\varphi) = \gamma + \varphi - \dbar^\star \cG (\dbar \varphi) = \gamma + \varphi,$$
where $\gamma$ is harmonic.  Since $\varphi =0$ on $E$, we have $L_{I_{1}} (\varphi)=0$. 

\noindent We obtain
$$\sum_{I_k \subset I_{k-1} \subset \cdots \subset I_0} \int_{E_{I_k}} \dbar^{\star_{I_k}} \cG_{I_k} (\Res_{E_{I_k}} (\wh \Theta)) \wedge L_{I_{k-1}} \circ L_{I_{k-2}} \circ \cdots \circ \dbar (L_{I_{0}}  (\varphi))$$
$$= \sum_{I_k \subset I_{k-1} \subset \cdots \subset I_0} \int_{E_{I_k}} \dbar^{\star_{I_k}} \cG_{I_k} (\Res_{E_{I_k}} (\wh \Theta)) \wedge L_{I_{k-1}} \circ  \cdots \circ L_{I_{1}} (\gamma + \varphi) =0.$$
This completes the proof of Lemma \ref{keylemm}.
\end{proof}	
\medskip

\begin{remark}{\rm It is important and challenging to refine further Theorem \ref{ddbar01} and 
solve the equation $\lambda= \dbar f$ \emph{with estimates}. A quick inspection of our arguments shows that one "only" needs to understand on what the constant $C$ 
in \eqref{check2} depends. This might not be the easiest thing to do, given the proof (by contradiction) for the existence of this constant.}\end{remark}
\bigskip

\subsection{Application: an injectivity theorem}
As a corollary of Theorem \ref{ddbar}, we obtain the following result.
\begin{thm}
Let $E =\sum E_i$ be a simple normal crossing divisor on a compact Kähler manifold $X$ and $F$ be a holomorphic line
bundle on $X$ with a smooth metric $h_F$ such that $i\Theta_{h_F} (F) \geq 0$.
Let $m\in\mathbb N^\star$ and consider a section $s\in H^0 (X,m F)$ such that the zero locus $s$ contains no lc centers of the lc pair $(X, E)$. Then the multiplication map induced by the tensor product with $s$
\begin{equation}\label{inject1}
s\otimes:  \qquad H^q (X, K_X + E + F) \rightarrow H^q (X, K_X +E + (m+1)F)
\end{equation}
is injective for every $q\in \mathbb N$.
\end{thm}

\begin{proof}
Let $s_E$ be the canonical section of $E$. After some desingularisation, by the assumption on the zero locus of $s$,  we can suppose that 
$$\Div (s) =\sum a_i F_i, $$ and $\sum F_i + E$ is snc. Indeed, consider a birational map $\pi:\wh X\to X$ such that the inverse image of 
$E+ \Div(s)$ is snc. Given a $(n, q)$-form $\lambda$
with values in $E+F$, the $\pi$ inverse image of $\displaystyle \frac{\lambda}{s_E}$ is a $(n, q)$-form with log-poles along the support of the inverse image of $E$, i.e. we have 
\begin{equation}
\pi^\star \left(\frac{\lambda}{s_E}\right)= \frac{\wh \lambda}{s_{\wh E}}
\end{equation}
where $\wh \lambda$ is a $(n, q)$ form on $\wh X$ with values in $\wh E+ \pi^\star(F)$. Another way to express this is the familiar equality
\begin{equation}
G+ \pi^\star(K_X+ E)= K_{\wh X}+ \wh E
\end{equation}
where $G$ is effective and the components of $\wh E$ are projecting onto the
log-canonical centers of $(X, E)$ (since we are only blowing up submanifolds which are transverse to $E$). In particular, the components of the support of $\wh E$ is disjoint from the inverse image of $\Div(s)$. 
\smallskip

\noindent Now if we are able to show that $\wh \lambda$ is 
$\dbar$-exact, then we write
\begin{equation}
\pi^\star \left({\lambda}\right)= \frac{\wh \lambda}{s_{\wh E}}\otimes \pi^\star(s_E).
\end{equation}
The RHS is holomorphic, since $\wh E$ is reduced and its support is
contained in the inverse image of $E$. It follows that the inverse image of $\lambda$ is zero in cohomology and then 
the same is true for $\lambda$.
\bigskip
	
\noindent We suppose for simplicity that $\Div (s)= aF_1 +bF_2$ (the general case follows from the same argument, by using the remark after Lemma \ref{reduceklt}). Let $s_i$ and $s_E$
be the canonical section of $F_i$ and of $E$, respectively.
One of the important points in our proof is to construct a sequence 
$(a_k, b_k)\in \mathbb N^2$ with the following properties.
\begin{enumerate}
\smallskip
  
\item[(1)] We have $(a_0, b_0)= (0,0), (a_N, b_N)= (a, b)$ and for each $k\geq 0$ the successive differences $a_{k+1}-a_k, b_{k+1}-b_k$ are elements of the set $\{0,1\}$.
\smallskip
  
\item[(2)] The bundle $F_k:= F+ a_k F_1 +b_k F_2$ admits a metric $\displaystyle h_k$ whose local weights can be written as
$$\varphi_k= c\varphi_F+  c_1\log|f_1|^2 + c_2\log|f_2|^2 $$
where $h_F= e^{-\varphi_F}$ is the non-singular, semi-positively curved metric on $F$, the constant $c$ is positive and $c_i\in ]0, 1[$ for $i=1, 2$.     
\end{enumerate}
Such sequence will be constructed in Lemma \ref{reduceklt}. We denote by
\begin{equation}
s_k:= s_1 ^{(a_{k+1}-a_k)} \otimes  s_2 ^{(b_{k+1}-b_k)}  \in H^0 (X, (a_{k+1}-a_k)F_1 +(b_{k+1}-b_k)F_2) .
\end{equation}
\smallskip

\noindent The property (1) above shows that \eqref{inject1} can written as composition of the maps
\begin{equation}\label{first}
s_k : H^q (X, K_X + E + F_k) \rightarrow H^q (X, K_X +E+ F_{k+1}).
\end{equation}
Thus, in order to conclude it is sufficient to prove that \eqref{first} is injective for every $k$. 	
\bigskip
		
\noindent Let $\lambda$ be a smooth form in  $H^q (X, K_X + E + F_k)$ such that 
$$\lambda \otimes s_k  =0 \in H^q (X, K_X +E+ F_{k+1}).$$ 
In other words there exists a smooth form $\alpha\in C^\infty _{(n,q-1)} (X,E+ F_{k+1})$ such that 
$$s_k \otimes \lambda = \dbar \alpha.$$
By property (2) we can equip $F_k= F + a_k F_1 +b_k F_2$ with the metric $h_k$. Now since the weight of the
metric $h_k$ contains a singular part $c_i\log|f_i|^2$ for some rational number $0<c_i <1$ and it follows from Lemma \ref{killpole} below that we can find forms $\beta_1, \beta_2$ with log poles along $E+F$ and $\theta$ with log poles along  $E$, such that the equality
\begin{equation}
 \frac{\lambda}{s_E} = \dbar \Big(\frac{\alpha}{s_{E}s_k}\Big) =
 D' _{h_k} (\beta_1) + \dbar \theta + i\Theta_{h_k} \wedge \beta_2
\end{equation} 
holds true on $X\setminus (E +F_1 +F_2)$. 
Therefore 
\begin{equation}\label{1ad}
	\frac{\lambda}{s_E} - \dbar \theta  = D' _{h_k} (\beta_1) + i\Theta_{h_k} \wedge \beta_2 
\end{equation} 
in the complement of $E +F_1 +F_2$.  Since $\theta$ is smooth near a generic point of $F_1 +F_2$, the forms on the LHS of \eqref{1ad} are non-singular at the generic point of $F_1+ F_2$, we are precisely in the framework of Theorem \ref{ddbar}.  
\smallskip

\noindent We infer  that $$\displaystyle \lambda - s_E\dbar \theta$$ is $\dbar$-exact. Since $\theta$ has by construction only simple poles along $E$, $\dbar (\theta) \otimes s_E =\dbar (\theta\cdot s_E)$ is $\dbar$-exact.  Then $\lambda$ is $\dbar$-exact and \eqref{first} is injective. The theorem is thus proved, modulo the two lemmas below.
\end{proof}
\medskip

\noindent To complete our proof, we establish next the auxiliary statements we used; we start with the existence of forms $\beta_i, \theta$.
\begin{lemme}\label{killpole}
Let $F_1 +F_2 +E$ be a snc divisor on $X$ and let $L$ be a holomorphic line bundle on $X$ with a possible singular metric $h_L$ such that 
$$i\Theta_{h_L} (L) = c_1 [F_1] + c_2 [F_2] + \theta_L$$
where $c_1, c_2$ are non zero and $\theta_L$ is smooth. Let $A$ be a $L$-valued smooth $(n,0)$-form on $X\setminus (E+F_1+F_2)$ with log poles along $E+F_1 +F_2$. Then we can find 
forms $\beta_1, \beta_2$ with at most log poles along $E+F_1 +F_2$, and $\theta$ with at most log poles along $E$, such that the equality
$$\dbar A= \dbar \theta +
D' _{h_L} (\beta_1) + i\Theta_{h_L} \wedge \beta_2 $$ 
holds in the complement of the support of $E +F_1 +F_2$.
\end{lemme}	

\begin{proof}
Let $\{\Omega_\alpha\}_\alpha$ be a cover of $X$ and let $\theta_\alpha$ be a partition of unity with respect to the covering. We suppose that $F_1 +F_2$ on $\Omega_\alpha$ are defined by $z_{1, \alpha} \cdot z_{2, \alpha}=0$.  
We define a smooth vector field 
$$V_1 := \sum_\alpha \theta_\alpha \frac{\partial}{\partial z_{1, \alpha}}\in C^\infty (X , T_X) .$$ 
Now on any local open set $\Omega\subset X$, we suppose that $F_1 +F_2$ is defined by $z_1 \cdot z_2 =0$. Then $V_1$ is of type:
$$V_1 |_{\Omega} = (z_1 + \O((z_1)^2 ) \frac{\partial}{\partial z_1} + \O(z_1z_2) \frac{\partial}{\partial z_2}+ \sum_{i\geq 3} a_i \frac{\partial}{\partial z_i}.$$
Then there is a constant $c$ such that $\displaystyle A_1 :=A- c D' _{h_L} (V_1 \rfloor  A)$ is smooth on $X\setminus (E+F_2)$ and simple pole along $E+F_2$.   
Moreover, by construction, $V_1 \rfloor  A$ has no poles along $F_1$ and it has at most log poles along $E+ F_2$. Note that we used the fact that $A$ is a $(n,0)$-form where $n =\dim X$.
\smallskip

\noindent Now 
$$\dbar (A)  =\dbar (A_1) + c \dbar (D' _{h_L} (V_1 \rfloor  A))  
=\dbar (A_1)  + D'_{h_L} \dbar (c V_1 \rfloor  A) + i\Theta_{h_L} (L) \wedge (c V_1 \rfloor  A)$$
pointwise on $X\setminus (E +F_1 +F_2)$.
By construction, $\dbar (c V_1 \rfloor  A) $ and $c V_1 \rfloor  A$ have log poles $E +F_2$. 
\smallskip

\noindent We repeat the same argument for $A_1$ to remove its pole along $F_2$. The resulting forms are $\theta$ with log poles along $E$ and $\beta_1, \beta_2$ with log poles along $E+ F$ so our lemma proved.
\end{proof}
\medskip

\noindent Finally, the construction of the sequence $(a_k, b_k)_{k=0,\dots, N}$ with the two properties (1) and (2) can be done as follows.
\begin{lemme}\label{reduceklt}
Let $F$ be a holomorphic line bundle on $X$ and let $m\in \mathbb N^\star$. Let $s\in H^0 (X, m F)$  such that $\Div (s) = a F_1 +b F_2$ for some $a, b \in \mathbb N^\star$. 
We consider a sequence $\{(a_k, b_k)\}_{k=1}^N \subset \mathbb N^2$ constructed as  follows:  
	
\begin{enumerate}

\item[\rm (i)] We define $(a_0, b_0)= (0,0)$. For $k\geq 1$, if $(a_k, b_k)= (a,b)$, then we stop at this step. If not, we take $c_k := \min\{\frac{a_k}{a}, \frac{b_k}{b} \}$.

\item[\rm (ii)] If $\frac{a_k}{a} > c_k$, then we set $a_{k+1} :=a_k$ and if
 $\frac{a_k}{a} =c$, we define $a_{k+1} :=a_k+1$.

\item[\rm (iii)] We do the same for $b_{k+1}$: if $\frac{b_k}{b} >c_k $, then
 $b_{k+1} :=b_k$ and if $\frac{b_k}{b} =c$, then we set $b_{k+1} :=b_k+1$.
\end{enumerate}	

\noindent Then for every $k\in \mathbb N$, we have 
\begin{equation}\label{mainlem}
0 = \min \{a_k - c_k a, b_k -c_k b\}	\leq  \max \{a_k - c_k a, b_k -c_k b\} <1 .
\end{equation}
In particular, there exist $\displaystyle c_{1,k}, c_{2,k} \in ]0,1[ \cap \mathbb Q$ and $d_k \in \mathbb Q^+$ such that
	\begin{equation}\label{kltpair}
		F+ a_k F_1 + b_k F_2 \equiv d_k F + c_{1,k} F_1 +c_{2,k} F_2.
	\end{equation}
\end{lemme}	

\begin{proof}
We prove \eqref{mainlem} by induction.  By the definition of $c_k$, we have 
	$$ \min \{a_k - c_k a, b_k -c_k b\} =0$$ 
so it is enough to establish the strict inequality for $\max$.

\noindent Assume that \eqref{mainlem} holds for $k$. Note first that by construction, $c_{k+1} > c_k$. There are three cases to consider.
\smallskip

\noindent $\bullet$ $c_k = \frac{a_k}{a}= \frac{b_k}{b} $.  Then $a_{k+1} -c_k a = b_{k+1} - c_k b =1$. As $c_{k+1}> c_k$, we know that \eqref{mainlem} holds for $k+1$.	
\smallskip

\noindent $\bullet$ $c_k = \frac{a_k}{a} < \frac{b_k}{b} $.  Then $a_{k+1} =a_k +1 $ and $b_{k+1} =b_k $. Therefore 
$$a_{k+1} - c_{k+1} a = a_k +1 - c_{k+1} a  = 1+  (c_k - c_{k+1} )a  <1 ,$$
and
$$b_{k+1} - c_{k+1} b =  b_k - c_{k+1} b < b_k - c_{k} b  <1 .$$
Then \eqref{mainlem} holds for $k+1$.	
\smallskip

\noindent $\bullet$ $c_k = \frac{b_k}{b} < \frac{a_k}{a} $.  We argue as in the second case. 
	\smallskip
	
\noindent For the second part of the lemma, thanks to \eqref{mainlem}, we have
	$$F+ a_k F_1 + b_k F_2 \equiv (1+ c_k m) F + (a_k - c_k a)F_1 +(b_k -c_k b) F_2. $$
	Let $\ep \in \mathbb Q^+$ small enough. Then 
	$$F+ a_k F_1 + b_k F_2 \equiv (1+ c_k m -\ep m ) F + (a_k - c_k a + \ep a)F_1 +(b_k -c_k b + \ep b) F_2 .$$
	Thanks to \eqref{mainlem}, for $\ep$ small enough, we know that 
	$(a_k - c_k a + \ep a)$ and $(b_k -c_k b + \ep b) \in ]0,1[$.  Moreover $(1+ c_k m -\ep m ) >0$.
	Then \eqref{kltpair} is proved.
\end{proof}	
\medskip

\noindent In general $(s=0)$ may have several components,
$\Div (s) = \sum_{1 \leq i\leq s} a_i F_i$. The construction of the sequence $\{ (a_{1,k},\cdots, a_{s,k})\}_{k\in\mathbb N}$ is absolutely identical.	

Indeed, suppose that $(a_{1,k},\cdots, a_{s,k})$ has already been constructed, for some $k\geq 1$. The next element $(a_{1,k+1},\cdots, a_{s,k+1})$ of our sequence is obtained as follows: 
	let $$c_k = \min \Big\{\frac{a_{1,k}}{a_1}, \frac{a_{2,k}}{a_2},\cdots, \frac{a_{s,k}}{a_s}\Big\}.$$ We define 
	$$a_{i, k+1} := a_{i,k}+1$$ in case $\displaystyle \frac{a_{i,k}}{a_i} =c_k$, and 
	$$a_{i, k+1} := a_{i,k}$$
if $\frac{a_{i,k}}{a_i} > c_k.$ One can easily verify by induction that for every $k\in \mathbb N$ the relation 
	\begin{equation}\label{kltpairnew}
		0 = \min \{a_{i,k} - c_k a_i \}_{i=1}^s	\leq  \max \{a_{i,k} - c_k a_i \}_{i=1}^s	 <1
	\end{equation}	
holds, precisely as we did in the case $s=2$. 
        $$$$

\section{Appendix}

\noindent We detail here the calculations needed for the evaluation of the limit \eqref{vanres48} (and discussion of the examples). They are rather long, but after all, nothing more than Taylor expansions. In connection to these topics, we are referring to the very nice article \cite{Mats} and the references therein.

\medskip

\noindent As a "warm-up" we consider the following particular case.
\smallskip

\noindent $(\diamondsuit)$ \emph{Assume that we can find a smooth metric $h_0$ on $\O(Y)$ whose local weights $\varphi_0$ are \emph{independent} on the
  $z_1,\dots, z_r$, where the
  $\displaystyle (z_i)_{i=1,\dots,n}$ are coordinates such that $\varphi_L(z)= \sum q_i\log|z_i|^2+ \varphi_{L, 0}.$}
\smallskip

\noindent In such case we argue as follows. 
Consider the following measure 
\begin{equation}\label{snc52}
d\lambda_{f, q}(z):= \frac{f(z)}{\prod_{i=1}^r z_i^{k_i}}\frac{d\lambda(z)}{\prod_{i=1}^r |z_i|^{2- 2\delta_i}}
\end{equation}
where the $\delta_i$ are strictly positive reals and $k_i$ are positive integers. 
\smallskip

\noindent Thanks to $(\diamondsuit)$ we immediately see that the quantity 
\begin{equation}\label{snc141}
\int_{(\CC^n, 0)}\rho_\ep'\left(\prod |z_i|^{2}e^{-\varphi_0(z')}\right)d\lambda_{f, q}(z) 
\end{equation}
where $z'=(z_{r+1},\dots z_n)$
tends to zero, regardless to the size of the multiplicities $q_i$. This is the case, since  
$f$ is smooth, so \eqref{snc141} amounts to the evaluation of
\begin{equation}\label{snc131}
\int_{(\CC^n, 0)}g(z')\rho_\ep'\left(\prod |z_i|^{2}e^{-\varphi_0(z')}\right)\frac{\prod_{i=1}^r z_i^{m_i}\ol z_i^{m_{\ol i}}}{\prod_{i=1}^r z_i^{k_i}}\frac{d\lambda(z)}{\prod_{i=1}^r |z_i|^{2- 2\delta_i}}   
\end{equation}
where $g$ is a smooth function only depending of the variables $z'$. 
\smallskip

\noindent The integral \eqref{snc131} is \emph{equal} to zero -for any $\ep> 0$ - unless we have
\begin{equation}
m_i- k_i= m_{\ol i}\geq 0.  
\end{equation}
But for such exponents the limit as $\ep\to 0$ of \eqref{snc131} is equal to zero, because the pole of order $k_i$ disappears, and $\delta_i>0$.
\medskip

\subsubsection{Taylor expansions} In general, i.e. in the absence of the assumption $(\diamondsuit)$, our arguments are virtually the same, modulo a few technicalities that we discuss next.

We will use 
Taylor formula with integral reminder. In order to illustrate the type of
statement we are after, we first start with a simple case.

\noindent Let $\phi$ be a smooth function defined on $\R$, and consider 
\begin{equation}\label{snc291}
f(x_1, x_2):= \phi(x_1+ x_2)
\end{equation}
a function defined say in a ball centred at the origin in $\R^2$.

\noindent The Taylor expansion of $f$ with respect to $x_1$ reads as follows
\begin{align}
  f(x_1, x_2)= & \label{snc311} \phi(x_2)+ x_1\phi'(x_2)+\dots+ \frac{x_1^k}{k!}\phi^{(k)}(x_2)\\
  + & \label{snc321} \frac{x_1^{k+1}}{(k+1)!}\int_0^1\phi^{(k+1)}(tx_1+ x_2)(1-t)^kdt.\\
\nonumber      
\end{align}
Next we expand each of the functions involved in \eqref{snc311} and \eqref{snc321} with respect to $x_2$, and we have:
\begin{align}
f(x_1, x_2)= & \label{snc331} \sum_{0\leq m_1, m_2\leq k}\frac{x_1^{m_1}x_2^{m_2}}{m_1! m_2!}\phi^{(m_1+ m_2)}(0)\\
  + & \label{snc341} \sum_{0\leq m_1\leq k}\frac{x_1^{m_1}x_2^{k+1}}{(k+1)!m_1!}\int_0^1\phi^{(m_1+k+1)}(tx_2)(1-t)^kdt \\
+ & \label{snc351} \sum_{0\leq m_2\leq k}\frac{x_1^{k+1}x_2^{m_2}}{(k+1)!m_2!}\int_0^1\phi^{(k+1+ m_2)}(tx_1)(1-t)^kdt\\
  + & \label{snc361} \frac{x_1^{k+1}x_2^{k+1}}{(k+1)!^2}\int_0^1\int_0^1\phi^{(2k+2)}(t_1x_1+ t_2x_2)(1-t_1)^k(1-t_2)^kdt_1dt_2.\\
\nonumber      
\end{align}

\noindent  Now we remark that in the expressions \eqref{snc331}-\eqref{snc361} we have three type of terms, according to the exponents of $x_1$ and $x_2$, respectively.
\begin{enumerate}

\item[\rm (1)] If both exponents $m_1$ and $m_2$ are smaller than $k$, then
  the function $\phi^{(m_1+ m_2)}(0)$ is constant with respect to both variables.

\item[\rm (2)] If, say, $m_1\leq k$ and $m_2= k+1$, then the function in question is only depending on $x_2$. The same is of course true if the roles of $x_1$ and $x_2$ are exchanged.

\item[\rm (3)] If both exponents $m_1$ and $m_2$ are equal to $k+1$, then the function in \eqref{snc361} is smooth (and in general, depending on both variables).
\end{enumerate}
\smallskip

\noindent This is not limited to the case of two variables, and the result is the same: the function which multiplies the monomial
\begin{equation}\label{snc371}
x_1^{m_1}\dots x_r^{m_r}
\end{equation}
is independent on $x_i$ if we have $m_i\leq k$, for each $i=1,\dots, r$
\medskip

\noindent Coming back to our problem, we will use a variation of this type of considerations in order to obtain an expansion which is adapted to our context. 

\smallskip Assume that the section $s$ of $\O(Y)$ corresponds to the monomial 
$z_1\dots z_r$ when restricted to $\Omega$
\begin{equation}\label{snc411}
s|_\Omega\simeq z_1\dots z_r.
\end{equation}  

\noindent The function we want to expand is
\begin{equation}\label{snc4011}
\rho_\ep\Big(\log\log\frac{1}{|s|^2e^{-\phi(z)}}\Big).
\end{equation}
To this end, remark that we have 
\begin{equation}\label{snc40111}
\log\frac{1}{|s|^2e^{-\phi(z)}}= \log\frac{1}{|s|^2e^{-\phi_{\wh 1}{(z)}}}+ z_1\phi_1(z)+ \ol {z_1}\phi_{\ol 1}(z)
\end{equation}
for some functions $\phi_1, \phi_{\ol 1}$, where $\phi_{\wh 1}(z):= \phi(0, z_2,\dots,z_n)$.

\noindent We consider
a function $\tau_1$ defined by the formula
\begin{equation}
  \tau_1(z):=  \log\left(1+ \frac{z_1\phi_1(z)+ \ol {z_1}\phi_{\ol 1}(z)}{\log\frac{1}{|s|^2e^{-\phi_{\wh 1}(z)}}}\right),
\end{equation}
and then we have 
\begin{equation}\label{snc4021}
\log\log\frac{1}{|s|^2e^{-\phi(z)}}= \log\log\frac{1}{|s|^2e^{-\phi_{\wh 1}{(z)}}}+ \tau_1(z).
\end{equation}
\medskip

\noindent In order to simplify the writing, we set the following notations, valid throughout the current appendix.
\smallskip

\noindent {\bf Conventions.}

\noindent $\bullet$ \emph{We will systematically denote by
  $\displaystyle a_{\wh {1\dots p}}$ any smooth function independent of the set of variables $z_1,\dots, z_p$ and their conjugates.}
\smallskip

\noindent $\bullet$ \emph{We use the same notation e.g. $a, b, \phi, ...$ for
functions which are not necessarily identical, but they share similar properties, which will be clearly specified.   
}

\smallskip

\noindent Then we have 
\begin{align}
\mu_\ep(z)= & \label{snc161}\rho_\ep\Big(\log\log\frac{1}{|s|^2e^{-\phi_{\wh 1}{(z)}}}\Big)  \\
 + & \label{snc171}\rho_\ep'\Big(\log\log\frac{1}{|s|^2e^{-\phi_{\wh 1}(z)}}\Big)\tau_1(z)+
\dots   \\+ & \rho_\ep^{(k)}\Big(\log\log\frac{1}{|s|^2e^{-\phi_{\wh 1}(z)}}\Big)\frac{\tau^k_1(z)}{k!}\\
  + &  \label{snc181}\frac{1}{k!}\tau_1(z)^{k+1}\int_0^1\rho_\ep^{(k)}\Big(t\tau_1(z)+
 \log\log\frac{1}{|s|^2e^{-\phi_{\wh 1}(z)}}\Big)(1- t)^kdt.\\
\nonumber  
\end{align}

\noindent We consider next the quantity 
\begin{equation}
\frac{z_1\phi_1(z)+ \ol {z_1}\phi_{\ol 1}(z)}{\log\frac{1}{|s|^2e^{-\phi_{\wh 1}(z)}}}
\end{equation}
involved in the expression of the function $\tau_1$. It can be rewritten as
\begin{equation}
\frac{z_1\phi_1(z)+ \ol {z_1}\phi_{\ol 1}(z)}
{\log\frac{1}{|s|^2e^{-\phi(z')}}+ z_2a_2(z)+ \ol z_2a_{\ol 2}(z)+\dots+ z_ra_r(z)+ \ol z_ra_{\ol r}(z)}
\end{equation} 
where $a_i, a_{\ol i}$ are smooth, and $\displaystyle z':=(z_{r+1},\dots, z_n).$

\noindent In conclusion, we have 
\begin{equation}\label{snc201}
\tau_1(z)= \Psi_1\left(z_1, w_1, \dots, z_r, w_r, z'\right)
\end{equation}
where \begin{equation}\label{snc201+}\displaystyle w_i:= \frac{z_i}{\log\frac{1}{|s|^2e^{-\phi(z')}}}\end{equation}
for $i=1,\dots,r$, and $\Psi_1$ is smooth and belonging to the ideal generated by $w_1, \ol w_1$.
\smallskip

\noindent Next we iterate this: in the expressions \eqref{snc161}--\eqref{snc181}, we use the fact that  
\begin{equation}\label{snc421}
\log{|s|^2e^{-\phi_{\wh 1}(z)}}= \log{|s|^2e^{-\phi_{\wh{12}}(z)}}+ z_2\phi_{\wh 12}(z)+ \ol z_2\phi_{\wh 1\ol 2}(z)
\end{equation}
where $\displaystyle \phi_{\wh 12}$ and $\displaystyle \phi_{\wh 1\ol 2}$
are independent of $z_1$.

\noindent We define
\begin{equation}
  \tau_2(z):=  \log\left(1+ \frac{z_2\phi_{\wh 12}(z)+ \ol z_2\phi_{\wh 1\ol 2}(z)}{\log\frac{1}{|s|^2e^{-\phi_{\wh{12}}(z)}}}\right),
\end{equation}
and consider the Taylor expansion of $\rho_\ep$ and its derivatives up to order $k$
in \eqref{snc161}-\eqref{snc171}.
\smallskip

\noindent For \eqref{snc181} we do the following. First we remark that we can write
\begin{equation}\label{snc231}
\tau_1(z)= a_{\wh 2}(z,w)+ b(z,w)
\end{equation}
where the notations in
\eqref{snc231} are as indicated below.
\smallskip

\noindent $\bullet$ As above, we have $\displaystyle w_i:= \frac{z_i}{\log\frac{1}{|s|^2e^{-\phi(z')}}}$. 
\smallskip

\noindent $\bullet$ The function $a_{\wh 2}$ is smooth, and it only depends on $z_1, w_1, z_3,\dots, z_n$. 
\smallskip

\noindent $\bullet$ The function $b$ belongs to the ideal generated by $z_2, \ol z_2$.
\medskip

\noindent By \eqref{snc421} we have 
$$\log\log\frac{1}{|s|^2e^{-\phi_{\wh 1}(z)}}= \log\log\frac{1}{|s|^2e^{-\phi_{\wh {12}}(z)}}+ b(z, w)$$
with the same properties as in the third bullet above.
\medskip

Then the function inside \eqref{snc181}
can be written as follows
\begin{align}
t\tau_1(z)+ \log\log\frac{1}{|s|^2e^{-\phi(z_2, z')}}= & ta_{\wh 2}(z, w)+ \log\log\frac{1}{|s|^2e^{-\phi_{\wh {12}}(z)}} \\ + & b_0(z, w)+ tb_1(z, w). \nonumber \\
\nonumber
\end{align}
and we expand the function $\rho_\ep^{(k)}$ in \eqref{snc181} at
$$ta_{\wh 2}(z, w)+ \log\log\frac{1}{|s|^2e^{-\phi_{\wh {12}}(z)}}.$$
\medskip

\noindent After this second step, we see that the function $\mu_\ep$ can be written as sum of terms of the following type.
\begin{equation}\label{snc431}
\tau_1^{m_1}\tau_2^{m_2}\rho_\ep^{(m)}\Big(
 \log\log\frac{1}{|s|^2e^{-\phi_{\wh {12}}(z)}}\Big)
\end{equation}
where $m:= m_1+ m_2$ and both $m_i$ are smaller or equal than $k$,  
\begin{equation}\label{snc441}
\tau_1^{m_1}\tau_2^{k+1}\int_0^1\rho_\ep^{(k+1+ m_1)}\Big(t\tau_2(z)+
 \log\log\frac{1}{|s|^2e^{-\phi_{\wh {12}}(z)}}\Big)(1- t)^kdt 
\end{equation}
where $m_1\leq k+1$ together with
\begin{equation}\label{snc561}
\tau_1^{k+1}\tau_2^{m_2}\int_0^1\rho_\ep^{(k+1+ m_2)}\Big(t a_{\wh 2}(z)+
 \log\log\frac{1}{|s|^2e^{-\phi_{\wh {12}}(z)}}\Big)(1- t)^kdt
\end{equation}  
(here $\tau_2$ is not the same as in \eqref{snc441}, but it belongs to the ideal generated by $z_2$ and $\ol z_2$, so we are using the same notation).

\noindent Finally, the last term is 
\begin{equation}\label{snc4501}
\int_0^1\int_0^1\rho_\ep^{(2k+2)}\Big(t_1b_1+ t_2b_2+ t_1t_2b_3
+ \log\log\frac{1}{|s|^2e^{-\phi_{\wh {12}}(z)}}\Big)(1- t)^kdt. 
\end{equation}
multiplied with $\tau_1^{k+1}\tau_2^{k+1}$, where $\displaystyle (1- t)^kdt:=
(1- t_1)^k(1- t_2)^kdt_1dt_2$.
\medskip

\noindent The conclusion is that in the resulting expression all the terms except for the RHS of \eqref{snc161}
have support contained in the domain
\begin{equation}\label{snc221}
  \frac{1}{\ep}- C< \log\log\frac{1}{|s|^2e^{-\phi_{\wh{12}}(z)}}< \frac{1}{\ep}+ C. 
\end{equation}
Moreover, the function
\begin{equation}\label{snc271}
  z\to \rho_\ep^{(m)}\Big(
 \log\log\frac{1}{|s|^2e^{-\phi_{\wh {12}}(z)}}\Big)
\end{equation}
in \eqref{snc431} only depends on $|z_1|, |z_2|,z_3,\dots, z_n$. The functions involved is the other expressions have similar properties, i.e. the one in \eqref{snc441} is independent of $z_1$, the one in \eqref{snc561} is independent of $z_2$. Note that the function in \eqref{snc4501} depends -in general- on the full set of variables, but the exponents of $\tau_i$ are $k+1$.
\medskip

\noindent We emphasize all this because integrals of the type
\begin{equation}\label{snc261}
\int_{\Omega}\frac{\tau_1^{m_1}\tau_2^{m_2}}{z_1^{q_1}z_2^{q_2}}\rho_\ep^{(m)}\Big(
 \log\log\frac{1}{|s|^2e^{-\phi_{\wh {12}}(z)}}\Big)\Psi(z,w)\frac{d\lambda(z)}{|z_1|^{2-2\delta_1}|z_2|^{2-2\delta_2}}
\end{equation}
are converging to zero as $\ep\to 0$ for any $m\geq 1$, for any $q_1, q_2$ positive integers, and $\delta_i> 0$ positive reals. In \eqref{snc261} we denote by $\Psi$ a smooth function.
\medskip

\noindent Anyway, we iterate this procedure, and we obtain the following statement.
\begin{proposition}\label{taylor1}
The function $\mu_\ep(\cdot)$ admits a Taylor expansion whose finite number of terms are of the following type
\begin{equation}
\tau_1^{m_1}(z, w)\dots \tau_r^{m_r}(z,w)\Phi_{\ep, m}(z, w)
\end{equation}
where:
\begin{enumerate}
\smallskip
  
\item[\rm (a)] the integer $r$ is given by \eqref{snc411}, and $\max(m_i)\leq k+1$;
\smallskip
  
\item[\rm (b)] for each $i=1,\dots ,r$ we have $\displaystyle w_i:=
  \frac{z_i}{\log\frac{e^{\varphi(z')}}{|z_1\dots z_r|^2}}$, where $z':= (z_{r+1},\dots , z_n)$;
\smallskip
  
\item[\rm (c)] the functions $\tau_i$ are smooth and they belong to the ideal generated by $w_i$ and $\ol w_i$;  
\smallskip
  
\item[\rm (d)] the function $\Phi_{\ep, m}(z, w)$
only depends on $|z_i|$ and $|w_i|$ if the corresponding exponent $m_i$ is smaller or equal to $k$,
for all $i=1,\dots ,r$;
\smallskip
  
\item[\rm (e)] if $m_1=\dots = m_r= 0$, then we have
\begin{equation}\nonumber
\Phi_{\ep, 0}(z, w)= \rho_\ep\Big(\log\log\frac{e^{\phi(z')}}{|z_1\dots z_r|^2}\Big);
\end{equation}
\smallskip
  
\item[\rm (f)] the support of the function
$\Phi_{\ep, m}(z, w)$ is contained in a set similar to \eqref{snc221} (i.e. rotationally symmetric with respect to each of the variables $z_1,\dots, z_r$) if $|m|\geq 1$. Note that even if $m_1=\dots = m_r= 0$, the support of the function $\Phi_{\ep, 0}$ is still rotationally symmetric.
\end{enumerate}
\end{proposition} 
\medskip

\noindent Thanks to Proposition \ref{taylor1}, the integrals we have to evaluate 
are of the following type
\begin{equation}\label{snc570}
\int_{(\CC^n, 0)}\rho_\ep^{(m)}(z)\frac{\Psi(z, w)}{\prod_{i=1}^r z_i^{b_i}}\frac{ d\lambda(z)}{\prod_{i=1}^r |z_i|^{2- 2\delta_i} 
\log |\prod_{i=1}^r z_i^{p_i}|^2 }
\end{equation}
where $m\geq 1$ and $z':= (z_{r+1},\dots ,z_n)$. We denote by $\Psi$ a smooth function, the $b_i$ in \eqref{snc570} are positive integers and 
$$\rho_\ep^{(m)}(z):= \rho_\ep^{(m)}\Big(\log\log\frac{e^{\phi(z')}}{|z_1\dots z_r|^2}\Big).$$ Now as $\ep\to 0$ the expression \eqref{snc570} clearly tends to zero, as we have already explained.
\medskip

\noindent To end with, we discuss now the examples presented in subsection \ref{subexample}.
\smallskip

\noindent Let $E= \sum E_i$ be a divisor on $X$ such that $E+Y$ is snc and let $\Theta$ be a $(p, q)$-form with values in $L$ and log poles along $E+Y$. We show next that the formula
\begin{equation}\label{app1}
\int_X\Theta\wedge \ol{\rho}:= \lim_{\ep\to 0} \sum_i\int_{U_i}\mu_\ep(w)\frac{\theta_i}{w_i^{mq}}\pi_i^\star\Theta\wedge \ol{\rho_i}
\end{equation}
defines a $(p, q)$-current with values in $(L, h_L)$. In \eqref{app1} we denote by $\rho$ a $(n-q, n-p)$-form with values in $L$, smooth in conic sense, so that $\rho_i$ on the LHS
is $\cC^\infty$.

\noindent In what follows we drop the index "i" and assume that locally the equations of the two divisors are
\begin{equation}\label{app2}
Y\cap V= (z_1\dots z_r= 0), \qquad E\cap V= (z_{r+1}\dots z_{r+ s}= 0).
\end{equation}
Then the integrals we have to evaluate are written as follows
\begin{equation}\label{app3}
\int_{U}\mu_\ep(w)\theta (|w|)\frac{f(w)d\lambda(w)}{\prod_{i=1}^r w_i^{b_i}\prod_{j=r+1}^{r+s}w_j}
\end{equation}
where $f$ is obtained by coefficients of $\rho_i$ multiplied with smooth functions (coefficients of the inverse image of $\Theta$). 
\smallskip

\noindent Now Proposition \ref{taylor1} shows that it would be enough to assume that 
\begin{equation}\label{app4}
\mu_\ep(w)= \rho_\ep\Big(\log\log\frac{e^{\phi(w'')}}{|w_1\dots w_r|^{2m}}\Big)
\end{equation}
where $w''= (w_{r+1},\dots, w_{n})$. We consider the Taylor expansion of $f$ with respect to 
the first $r$ variables and observe that for the integrals
\begin{equation}\label{app5}
\int_{U}\mu_\ep(w)\theta (|w|)\frac{\prod_{i=1}^r w_i^{\alpha_i}\ol w_i^{\beta_i}}{\prod_{i=1}^r w_i^{b_i}\prod_{j=r+1}^{r+s}w_j}d\lambda(w)
\end{equation}
is equal to zero as soon as for some index $i$ we have $\alpha_i+\beta_i\leq b_i-1$: indeed, the integral in \eqref{app5} is uniformly convergent, so our assertion follows by Fubini together with the fact that the integration domain is rotationally symmetric. 
\smallskip

\noindent Therefore, only the terms 
\begin{equation}\label{app6}
\int_{U}\mu_\ep(w)\theta (|w|)\frac{\prod_{i=1}^r w_i^{\alpha_i}\ol w_i^{\beta_i}f_{\alpha\beta}(w)}{\prod_{i=1}^r w_i^{b_i}\prod_{j=r+1}^{r+s}w_j}d\lambda(w)
\end{equation}
with $\alpha_i+ \beta_i= b_i$ for each $i$ "survive". But this type of integrals are bounded by 
$\displaystyle \sup_U |f_{\alpha\beta}|$, which in turn is bounded by the sup norm of a certain number of derivatives of $\rho_i$. 
\smallskip

\noindent Hence $\Theta$ induces a $(p, q)$-current with values in $(L, h_L)$, as claimed in 
\ref{subexample}. A last remark is that if $p= n$ and $\Theta$ has only log-poles along $E$, then the limit process above is not needed 
in case $q_i< 1$ because the quantity
$\displaystyle \frac{1}{w^{mq}}\pi^\star(dz_1\wedge\dots\wedge dz_n)$ is bounded.

\end{document}